\documentclass[11pt, twoside, leqno]{article}

\usepackage{amsmath}
\usepackage{amssymb}
\usepackage{mathrsfs}
\usepackage{amsthm}

\usepackage{indentfirst}
\usepackage{color}
\usepackage{txfonts}

\allowdisplaybreaks

\def\bint{{\ifinner\rlap{\bf\kern.30em--}
\int\else\rlap{\bf\kern.35em--}\int\fi}\ignorespaces}

\def\sbint{{\ifinner\rlap{\bf\kern.32em--}
\hspace{0.078cm}\int\else\rlap{\bf\kern.45em--}\int\fi}\ignorespaces}

\def\red{\color{red}}

\def\rr{\mathbb{R}}
\def\rn{\mathbb{R}^n}
\def\rnn{\mathbb{R}^{2n}}

\def\nn{\mathbb{N}}
\def\zz{\mathbb{Z}}

\def\ez{\epsilon}

\def\bz{\beta}

\def\wz{\widetilde}

\def\ls{\lesssim}

\def\fz{\infty}
\def\az{\alpha}

\def\cd{{\mathcal D}}
\def\ce{{\mathcal E}}
\def\cf{{\mathcal F}}

\def\ck{{\mathcal K}}

\def\cm{{\mathcal M}}

\def\co{{\mathcal O}}

\def\BMO{\mathop\mathrm{\,BMO\,}}

\def\CMO{\mathop\mathrm{\,CMO\,}}
\def\VMO{\mathop\mathrm{\,VMO\,}}
\def\XMO{\mathop\mathrm{\,XMO\,}}
\def\MMO{\mathop\mathrm{\,MMO\,}}
\def\Lpwa{{L_{w_1}^{p_1}(\rn)}}
\def\Lpwb{{L_{w_2}^{p_2}(\rn)}}
\def\Lpw{{L_w^p(\rn)}}
\def\xMO{\mathop\mathrm{\,X_1MO\,}}

\def\r{\right}
\def\lf{\left}

\def\r{\right}
\def\lf{\left}

\def\supp{{\mathop\mathrm{\,supp\,}}}

\def\loc{{\mathop\mathrm{\,loc\,}}}

\def\BMO{{\mathop\mathrm{\,BMO\,}}}

\def\eqref#1{(\ref{#1})}

\newtheorem{theorem}{Theorem}[section]
\newtheorem{lemma}[theorem]{Lemma}
\newtheorem{corollary}[theorem]{Corollary}
\newtheorem{proposition}[theorem]{Proposition}

\theoremstyle{definition}
\newtheorem{remark}[theorem]{Remark}

\newtheorem{question}[theorem]{Question}

\numberwithin{equation}{section}

\begin{document}
\title{\bf\Large XMO and Weighted Compact Bilinear Commutators
\footnotetext{\hspace{-0.35cm} 2020 {\it
Mathematics Subject Classification}. Primary 42B20; Secondary 42B35, 46E35, 47B47, 47A30. \endgraf
{\it Key words and phrases}. bounded mean oscillation, commutator,
compact bilinear operator, weight.
 \endgraf
This project is supported by the National
Natural Science Foundation of China  (Grant Nos.  11761131002, 11971058, 11671185,
11671039, 11871101 and 11871100).}}
\date{ }
\author{Jin Tao, Qingying Xue,
Dachun Yang\,\footnote{Corresponding
author, E-mail: \texttt{dcyang@bnu.edu.cn} / {\red July 5, 2020} / Final version.} \  and Wen Yuan}
\maketitle

\vspace{-0.8cm}

\begin{center}
\begin{minipage}{13cm}
{\small {\bf Abstract}\quad
To study the compactness of bilinear commutators
of certain bilinear Calder\'on--Zygmund operators
which include (inhomogeneous) Coifman--Meyer bilinear
Fourier multipliers and bilinear pseudodifferential operators
as special examples, Torres and Xue
[Rev. Mat. Iberoam. 36 (2020), 939--956]
introduced a new subspace of
BMO$\,(\mathbb{R}^n)$, denoted by XMO$\,(\mathbb{R}^n)$,
and conjectured that it is just
the space VMO$\,(\mathbb{R}^n)$ introduced by D. Sarason.
In this article, the authors give a negative answer to this conjecture
by establishing an equivalent characterization of XMO$\,(\mathbb{R}^n)$, which further
clarifies that XMO$\,(\mathbb{R}^n)$ is a proper subspace of VMO$\,(\mathbb{R}^n)$.
This equivalent characterization of XMO$\,(\mathbb{R}^n)$ is formally similar to the corresponding
one of CMO$\,(\mathbb{R}^n)$ obtained by A. Uchiyama,
but its proof needs some essential new techniques on dyadic cubes as well as some
exquisite geometrical observations. As an application, the authors also obtain a
weighted compactness result on such bilinear commutators, which optimizes the
corresponding result in the unweighted setting. }
\end{minipage}
\end{center}

\vspace{0.2cm}

\section{Introduction}

In a very recent article \cite{TX19}, to study the compactness of bilinear commutators
of certain bilinear Calder\'on--Zygmund operators
which include (inhomogeneous) Coifman--Meyer bilinear
Fourier multipliers and bilinear pseudodifferential operators
as special examples,
Torres and Xue introduced a new subspace of
BMO$\,(\mathbb{R}^n)$, denoted by XMO$\,(\mathbb{R}^n)$,
and conjectured that it is just
the space VMO$\,(\mathbb{R}^n)$ introduced by Sarason \cite{Sarason75}.
In this article, we give a negative answer to this conjecture
by establishing an equivalent characterization of XMO$\,(\mathbb{R}^n)$, which further
clarifies that XMO$\,(\mathbb{R}^n)$ is a proper subspace of VMO$\,(\mathbb{R}^n)$.
This equivalent characterization of XMO$\,(\mathbb{R}^n)$ is formally similar to the corresponding
one of CMO$\,(\mathbb{R}^n)$ obtained by Uchiyama \cite{Uchiyama78TohokuMathJ},
but its proof needs some essential new techniques on dyadic cubes as well as some
exquisite geometrical observations. As an application, we also obtain a
weighted compactness result on such bilinear commutators, which optimizes the
corresponding result of Torres and Xue \cite{TX19} in the unweighted setting.

In what follows, we use $L_{\rm c}^\fz(\rn)$ to denote the set of all
essentially bounded functions on $\rn$ with compact support.
The theory of commutators of pointwise multiplication with Calder\'on--Zygmund
operators has attracted lots of attentions and many works have been done
since Coifman et al. \cite{CoifmanRochbergWeiss76AnnMath} first studied the
boundedness characterization of the commutator $[b,T]$ which is defined by setting,
for any $f\in L_{\rm c}^\fz(\rn)$,
$$[b,T](f):=bT(f)-T(bf),$$
where  $T$ is any classical Calder\'on--Zygmund operator
with smooth kernel and $b\in\BMO(\rn)$.
Among those achievements are the celebrated boundedness and compactness results of
Coifman et al. \cite{CoifmanRochbergWeiss76AnnMath}, Cordes \cite{CO},
Uchiyama \cite{Uchiyama78TohokuMathJ} and Janson \cite{J} in the linear situation.
In \cite{Uchiyama78TohokuMathJ}, Uchiyama established a
characterization of $\CMO(\rn)$ (see Proposition \ref{CMO-char} below)
and used it to show that, for any given $p\in(1,\fz)$ and any
Calder\'on--Zygmund operator $T$ with rough kernel,
$[b,T]$ is compact on $L^p(\rn)$ if and only if $b$ is in $\CMO(\rn)$,
where $\CMO(\rn)$ denotes the closure in $\BMO(\rn)$ of infinitely differential
functions with compact support.

In the bilinear setting, recall that the boundedness on $L^p(\rn)$ of
the commutators of more general bilinear Calder\'on--Zygmund operators
with $b\in \BMO(\rn)$ was established by P\'erez and Torres \cite{PT}
for any given $p\in (1,\fz)$, and by Tang \cite{Tang} and Lerner et al. \cite{LOPTT09AM}
for any given $p\in (1/2, 1]$. The compactness in $L^p(\rn)$ of the commutators multiplying
functions in $\CMO(\rn)$ was demonstrated by  B\'enyi and Torres \cite{BT}
for any given $p\in (1,\fz)$, and by Torres et al. \cite{TXY} for any given $p\in (1/2, 1]$.
Moreover, Chaffee et al. \cite{CCHTW} showed that the compactness result
for certain homogeneous bilinear Calder\'on--Zygmund operators holds true
if and only if $b\in \CMO(\rn)$;
see also Remark \ref{compact-rem}(iv) below.
For more related works, we refer the reader to
\cite{CDLW19,DMX,H1,LNWW20,L20} and their references.

In order to investigate the possible versions in the bilinear setting of the
compactness result of Cordes \cite{CO}, Torres and  Xue in \cite{TX19} uncovered
two subspaces of $\BMO(\rn)$, which were denoted, respectively,
by $\MMO(\rn)$ and $\XMO(\rn)$. It is known that
$$\CMO(\rn) \subsetneqq \MMO(\rn) \subsetneqq \XMO(\rn) \subset \VMO(\rn),$$
where $\VMO(\rn) \subsetneqq \BMO(\rn)$ denotes the space of functions with
``vanishing mean oscillation".
The main results in \cite{TX19}  state that the compactness result still holds true
for the commutators of pointwise multiplication with certain bilinear Calder\'on--Zygmund
operators whenever $b\in \XMO(\rn).$ This means, of course, for the compactness of
these commutators, $b$ does not need to be in $\CMO(\rn)$.
It still works in a larger subspace $\XMO(\rn)$.

In what follows, let $\mathbb{N}:=\{1,\,2,...\}$, $\mathbb{Z}_+:=\mathbb{N}\cup\{0\}$,
$\mathbb{Z}_+^n:=(\mathbb{Z}_+)^n$ and $\mathbb{Z}_+^{3n}:=(\mathbb{Z}_+)^{3n}$.
In this article, we consider the following particular type
bilinear Calder\'on--Zygmund operator $T$, whose kernel $K$ satisfies
\begin{itemize}
\item[{\rm(i)}]The standard \emph{size} and \emph{regularity} conditions:
for any given multi-indices $\az:=(\az_1,\ldots,\az_{3n})\in\zz_+^{3n}$ with
$|\az|:=\az_1+\cdots+\az_{3n}\le1$, there exists a positive constant
$C_{(\az)}$, depending on $\az$, such that,
for any $x,\ y,\ z\in\rn$ with $x\neq y$ or $x\neq z$,
\begin{align}\label{sizeregular}
|D^\az K(x,y,z)|\le C_{(\az)} (|x-y|+|x-z|)^{-2n-|\az|}.
\end{align}
Here and thereafter,
$D^\az:=(\frac\partial{\partial x_1})^{\az_1}\cdots
 (\frac\partial{\partial x_{3n}})^{\az_{3n}}$.

\item[{\rm(ii)}]The additional decay condition:
there exists a positive constant $C$ such that,
for any $x,\ y,\ z\in\rn$ with $|x-y|+|x-z|>1$,
\begin{align}\label{decay}
|K(x,y,z)|\le C (|x-y|+|x-z|)^{-2n-2}
\end{align}
\end{itemize}
and, for any $f,\ g\in L_{\rm c}^\fz(\rn)$ and $x\notin\supp(f)\cap\supp(g)$,
$T$ is supposed to have the following usual representation:
\begin{align*}
T(f,g)(x)=\int_{\rnn}K(x,y,z)f(y)g(z)\,dy\,dz,
\end{align*}
here and thereafter, $\supp(f):=\{x\in\rn:\ f(x)\neq 0\}$.
The (inhomogeneous) Coifman--Meyer bilinear
Fourier multipliers and the bilinear pseudodifferential operators with certain symbols
satisfy the above conditions (see, for instance, \cite{TX19}).
Therefore, they are typical examples of the bilinear
Calder\'on--Zygmund operators as above.
We refer the reader also to \cite{BT03,BD13,CM78,GT02,GMNT,GMN,TX19}
for the boundedness and more history of multilinear Fourier multipliers
and pseudodifferential operators.
The original motivation of \cite{TX19} is to prove that, if the kernel
of the modified Calder\'on--Zygmund operator $T$ in the considered commutator
$[b,T]$ has some better decay properties than the classical one,
then  $[b,T]$ should be compact for $b$ being in a larger subspace of $\BMO(\rn)$
than $\CMO(\rn)$, which indeed proved true in \cite{TX19}.

Recall that the \emph{bilinear commutators} with single entries are
defined by setting, for any $f,\ g\in L_{\rm c}^\fz(\rn)$ and
$x\notin\supp(f)\cap\supp(g)$,
\begin{equation}\label{c1}
[b,T]_1(f,g)(x):=\lf(bT(f,g)-T(bf,g)\r)(x)
=\int_{\rnn}[b(x)-b(y)]K(x,y,z)f(y)g(z)\,dy\,dz
\end{equation}
and
\begin{equation}\label{c2}
[b,T]_2(f,g)(x):=\lf(bT(f,g)-T(bf,g)\r)(x)
=\int_{\rnn}[b(x)-b(z)]K(x,y,z)f(y)g(z)\,dy\,dz.
\end{equation}

We now need to introduce several subspaces of the space $\BMO(\rn)$.
Recall that
$$\CMO(\rn):=\overline{C_{\rm c}^\fz(\rn)\cap\BMO(\rn)}^{\BMO(\rn)}$$
and
$$\VMO(\rn):=\overline{C_{\rm u}(\rn)\cap\BMO(\rn)}^{\BMO(\rn)},$$
where $C_{\rm c}^\fz(\rn)$ denotes the set of all smooth functions on $\rn$ with
compact support and $C_{\rm u}(\rn)$ the set of all functions on $\rn$ with
uniform continuity. Here and thereafter, $\overline{\mathcal X}^{\BMO(\rn)}$
denotes the closure in $\BMO(\rn)$ of the set $\mathcal X$.

In what follows, we use $\vec 0_n$ to denote the origin of $\rn$
and, for any $\az:=(\az_1,\ldots,\az_n)\in\zz_+^n$, we let
$D^\az:=(\frac\partial{\partial x_1})^{\az_1}\cdots(\frac\partial{\partial x_n})^{\az_n}$.
We also use $C^\fz(\rn)$ to denote the set of all infinitely differentiable functions on $\rn$ and
$L^\fz(\rn)$ the set of all essentially bounded functions on $\rn$.
The spaces $\MMO(\rn)$ and $\XMO(\rn)$ in \cite{TX19}
were defined in the way that
$$\MMO(\rn):=\overline{A_\fz(\rn)}^{\BMO(\rn)},$$
where
$$A_\fz(\rn):=\lf\{b\in C^\fz(\rn)\cap L^\fz(\rn):
\,\,\forall\ \az\in\zz_+^n\setminus\{\vec 0_n\},
\lim_{|x|\to\fz}D^\az b(x)=0\r\},$$
and
$$\XMO(\rn):=\overline{B_\fz(\rn)}^{\BMO(\rn)},$$
where
$$B_\fz(\rn):=\lf\{b\in C^\fz(\rn)\cap\BMO(\rn):
\,\,\forall\ \az\in\zz_+^n\setminus\{\vec 0_n\},
\lim_{|x|\to\fz}D^\az b(x)=0\r\}.$$

Furthermore, we use the following set
$$B_1(\rn):=\lf\{b\in C^1(\rn)\cap\BMO(\rn):\,\,
\lim_{|x|\to\fz}|\nabla b(x)|=0\r\}$$
to define
$$\xMO(\rn):=\overline{B_1(\rn)}^{\BMO(\rn)},$$
where $C^1(\rn)$ denotes the set of all functions $f$ on $\rn$ whose gradients
$\nabla f:=(\frac{\partial f}{\partial x_1},\dots,\frac{\partial f}{\partial x_n})$
are continuous.
By the observation
$C_{\rm c}^\fz(\rn)\subset B_\fz(\rn)\subset B_1(\rn)\subset C_{\rm u}(\rn)$,
we easily conclude that
$$\CMO(\rn)\subset\XMO(\rn)\subset\xMO(\rn)\subset\VMO(\rn).$$
Moreover, it was shown in \cite{TX19} that
$$\CMO(\rn)\subsetneqq\MMO(\rn)\subsetneqq\XMO(\rn).$$
Meanwhile, an open question was posed by Torres and  Xue in \cite{TX19} as follows:

\begin{question}\label{question}
Which one of the following two possibilities
$$\XMO(\rn)\subsetneqq \VMO(\rn)\quad \mathrm{or}\quad \XMO(\rn)= \VMO(\rn)$$
holds true?
\end{question}

Torres and  Xue in \cite{TX19} conjectured that the latter
might be true. However, in this article, we show that the relationship
$\XMO(\rn)\subsetneqq \VMO(\rn)$ holds true, which gives a complete answer to
Question \ref{question}. Indeed, we have
$$\CMO(\rn)\subsetneqq\XMO(\rn)=\xMO(\rn)\subsetneqq\VMO(\rn),$$
where $\XMO(\rn)\supseteqq\xMO(\rn)$ is quite surprising.
To show this, we establish the following equivalent characterization,
which is the first main result of this article.
In what follows, the {symbol} $a\to0^+$ means that $a\in(0,\fz)$ and $a\to0$;
the {symbol} $Q$ means a cube that $Q$ has finite side length,
all its sides parallel to the coordinate axes,
but $Q$ is not necessary to be open or closed, and $Q+x:=\{y+x:\ y\in Q\}$ for any $x\in\rn$;
for any cube $Q\subsetneqq\rn$ and $f\in L_{\loc}^1(\rn)$
(the set of all locally integrable functions),
the \emph{mean oscillation} $\co(f;Q)$ is defined by setting
$$\co(f;Q):=\frac1{|Q|}\int_Q\lf|f(x)-\frac1{|Q|}\int_Q f(y)\,dy\r|\,dx.$$

\begin{theorem}\label{xMO-char}
The following statements are mutually equivalent:
\begin{itemize}
\item [{\rm(i)}] $f\in\xMO(\rn)$;
\item [{\rm (ii)}] $f\in\BMO(\rn)$ and enjoys the properties that
\begin{itemize}
\item [{\rm (ii)$_1$}] $$\lim_{a\to0^+}\sup_{|Q|=a}\co(f; Q)=0;$$
\item [{\rm (ii)$_2$}]for any cube $Q\subset\rn$, $$\lim_{|x|\to\fz}\co(f; Q+x)=0.$$
\end{itemize}
\item [{\rm (iii)}] $f\in\XMO(\rn)$.
\end{itemize}
\end{theorem}

As a consequence of Theorem \ref{xMO-char}, we have the following conclusion.
\begin{corollary}\label{xMO-coro}
$\xMO(\rn)=\XMO(\rn)\subsetneqq\VMO(\rn)$.
\end{corollary}

Thus, Corollary \ref{xMO-coro} completely answers the open question
asked by Torres and  Xue in \cite{TX19}.

In order to state another main result of this article,
we need to introduce a class of multiple weights.
Recall that, usually, a non-negative measurable function $w$ on $\rn$
is called a \emph{weight} on $\rn$.
For any given $\textbf{p}:=(p_1,p_2)\in(1,\fz)\times(1,\fz)$,
let $p$ satisfy $\frac1p=\frac1{p_1}+\frac1{p_2}$.
Following \cite{BDMT15},  we call $\textbf{w}:=(w_1,w_2)$ a
\emph{vector} $\textbf{A}_{\textbf{p}}(\rn)$ \emph{weight},
denoted by $\textbf{w}:=(w_1,w_2)\in \textbf{A}_{\textbf{p}}(\rn)$,
if
$$[\textbf{w}]_{\textbf{A}_{\textbf{p}}(\rn)}:=
\sup_Q\lf[\frac 1{|Q|}\int_Q w(x)\,dx\r]\lf\{\frac 1{|Q|}
\int_Q \lf[w_1(x)\r]^{1-p_1'}\,dx\r\}^{\frac p{p_1'}}
\lf\{\frac 1{|Q|}\int_Q \lf[w_2(x)\r]^{1-p_2'}\,dx\r\}^{\frac p{p_2'}}<\fz,$$
where $w:=w_1^{p/p_1}w_2^{p/p_2}$,
$\frac 1{p_1}+\frac 1{p_1'}=1=\frac 1{p_2}+\frac 1{p_2'}$
and the supremum is taken over all cubes $Q$ of $\rn$.
In what follows, for any given weight $w$ on $\rn$
and measurable subset $E\subseteqq\rn$,
the symbol $L^p_w(E)$ denotes the set of all measurable
functions $f$ on $E$ such that
$$\|f\|_{L^p_w(E)}:=\lf[\int_E |f(x)|^p w(x)\,dx\r]^\frac1p<\fz.$$

Now, we state our second main result of this article
on an application of $\XMO(\rn)$ as follows.

\begin{theorem}\label{compact-thm}
Let $\textbf{p}:=(p_1,p_2)\in(1,\fz)\times(1,\fz)$,
$p\in(\frac12,\fz)$ with $\frac1p=\frac1{p_1}+\frac1{p_2}$,
$\textbf{w}:=(w_1,w_2)\in \textbf{A}_{\textbf{p}}(\rn)$,
$w:=w_1^{p/p_1}w_2^{p/p_2}$, $b\in\XMO(\rn)$ and
$T$ be a bilinear Calder\'on--Zygmund operator whose kernel satisfies
\eqref{sizeregular} and \eqref{decay}. Then, for any $i\in\{1,2\}$,
the bilinear commutator $[b,T]_i$ as in \eqref{c1} or \eqref{c2}
is compact from $\Lpwa\times\Lpwb$ to $\Lpw$.
\end{theorem}

\begin{remark}\label{compact-rem}
We have the following comments towards the conclusions of Theorem \ref{compact-thm}.
\begin{itemize}
\item[{\rm(i)}] Although we state and prove Theorem \ref{compact-thm}
in bilinear case, indeed this theorem can be extended to linear
or multilinear case with notational complications and usual modifications.
For instance, if $b\in\XMO(\rn)$ and $T$ is a linear Calder\'on--Zygmund operator
whose kernel $K$ satisfies that, for any given
$\az:=(\az_1,\dots,\az_{2n})\in\zz_+^{2n}$ with
$|\az|:=\az_1+\cdots+\az_{2n}\le1$,
and any $x,\ y\in\rn$,
$$|D^\az K(x,y)|\le C_{(\az)} |x-y|^{-n-|\az|}$$
and, for any $x,\ y\in\rn$ with $|x-y|\ge1$,
$$|K(x,y)|\le C |x-y|^{-n-2},$$
where $D^\az:=(\frac\partial{\partial x_1})^{\az_1}\cdots
 (\frac\partial{\partial x_{2n}})^{\az_{2n}}$,
$C_{(\az)}$ and $C$ are some positive constants,
then $[b,T]$ is compact on $L^p_w(\rn)$ for any given $p\in(1,\fz)$ and $w\in A_p(\rn)$.
Furthermore, observe that the proof of Theorem \ref{compact-thm} mainly depends on
the boundedness of Calder\'on--Zygmund operators and the Hardy--Littlewood maximal operators.
Therefore, Theorem \ref{compact-thm} can also be extended to Morrey spaces;
see, for instance, \cite{TYY19}.

\item[{\rm(ii)}] The corresponding compactness result in
\cite[Theorem 1.1]{TX19} requires
the kernel $K$ satisfying both \eqref{sizeregular} and the following additional estimates:
for any given $\az\in\zz_+^{3n}$, with $|\az|\le1$, and for any given $N\in\{1,2,3\}$,
there exists a positive constant $C_{(\az, N)}$, depending on $\az$ and $N$, such that,
for any $|x-y|+|x-z|>1$,
\begin{align}\label{N=123}
|D^\az K(x,y,z)|\le C_{(\az, N)}(|x-y|+|x-z|)^{-2n-N}.
\end{align}
But, our assumption \eqref{decay} in Theorem \ref{compact-thm} only needs
$\az=\vec0_{3n}$ and $N=2$ in \eqref{N=123}.
Thus, in this sense, even the \emph{unweighted case} of Theorem \ref{compact-thm} also
optimizes and hence improves the corresponding result in \cite{TX19}.

\item[{\rm(iii)}]
B\'enyi et al. \cite[Theorem 1.1]{BDMT15} obtained the compactness of
weighted compact bilinear operators via $\CMO(\rn)$, which states that,
if $b\in\CMO(\rn)$ and $T$ is a bilinear Calder\'on--Zygmund operator
whose kernel $K$ satisfies \eqref{sizeregular},
then the bilinear commutators $\{[b,T]_i\}_{i=1}^2$
are compact from $\Lpwa\times\Lpwb$ to $\Lpw$.
From this and Proposition \ref{CMO-char} below, we deduce that
\begin{align}\label{1+3}
\begin{cases}
T {\rm\ satisfies\ } \eqref{sizeregular}\\
b {\rm\ satisfies\ (i),\ (ii)\ and\ (iii)\ of\ Proposition\ \ref{CMO-char}}
\end{cases}
\Longrightarrow
\{[b,T]_i\}_{i=1}^2\ {\rm are\ compact}.
\end{align}
On the other hand, by Theorems \ref{xMO-char} and \ref{compact-thm},
and Proposition \ref{CMO-char} below, we conclude that
\begin{align}\label{2+2}
\begin{cases}
T {\rm\ satisfies\ \eqref{sizeregular}{\rm\ and\ }\eqref{decay}}\\
b {\rm\ satisfies\ (i)\ and\ (ii)\ of\ Proposition\ \ref{CMO-char}}
\end{cases}
\Longrightarrow
\{[b,T]_i\}_{i=1}^2\ {\rm are\ compact}.
\end{align}
Therefore, in \eqref{1+3},
if we make an additional assumption \eqref{decay} on $T$,
and drop the condition (iii) of Proposition \ref{CMO-char} on $b$,
then it coincides with \eqref{2+2}. This is harmonious and reasonable.

Besides, \cite[Theorem 1.1]{BDMT15} requires $p:=\frac{p_1 p_2}{p_1+p_2}>1$
because they used the weighted Frech\'et--Kolmogorov theorem on $L^p_w(\rn)$
with $p\in(1,\fz)$. However, thanks to \cite[Theorem 1]{XYY18}, which is re-stated
as Lemma \ref{FK} below, we can optimize and hence improve this range into $p\in(\frac12,\fz)$
in Theorem \ref{compact-thm}.

\item[{\rm(iv)}]
Chaffee et al. \cite[Theorem 3.1]{CCHTW} proved that,
letting $p_1,\ p_2\in(1,\fz)$, $p:=\frac{p_1 p_2}{p_1+p_2}>\frac12$ and
$\{\mathcal{R}_j^k:\ j\in\{1,2\}\ {\rm and}\ k\in\{1,\dots,n\}\}$
be the \emph{bilinear Riesz transform} defined by setting,
for any given $k\in\{1,\dots,n\}$ and any $x:=(x_1,\dots,x_n)\in\rn$,
$$\mathcal{R}_1^k(f,g)(x)
:={\rm p.\,v.}\int_{\rr^{2n}}\frac{x_k-y_k}{(|x-y|^2+|x-z|^2)^{n+\frac12}}
f(y)g(z)\,dy\,dz$$
and
$$\mathcal{R}_2^k(f,g)(x)
:={\rm p.\,v.}\int_{\rr^{2n}}\frac{x_k-z_k}{(|x-y|^2+|x-z|^2)^{n+\frac12}}
f(y)g(z)\,dy\,dz,$$
then, for any $i,\, j\in\{1,2\}$ and $k\in\{1,\dots,n\}$,
$[b,\mathcal{R}_j^k]_i$ is compact from
$L^{p_1}(\rn)\times L^{p_2}(\rn)$ to $L^{p}(\rn)$
if and only if $b\in\CMO(\rn)$.
Moreover, as a bilinear counterpart of \cite[Theorem 2]{Uchiyama78TohokuMathJ},
Chaffee et al. \cite[Remark 3.2]{CCHTW} pointed out that
\cite[Theorem 3.1]{CCHTW} also holds true if the bilinear Riesz transform
is replaced by any more general bounded convolution bilinear operator with rough kernel
\begin{align}\label{roughKernel}
\frac{\Omega(\frac{(y,z)}{|(y,z)|})}{(|y|^2+|z|^2)^n},
\end{align}
where $(y,z)\in\rn\times\rn\setminus\{\vec 0_{2n}\}$, $\Omega$ is a homogeneous
function of degree zero defined on the unit sphere in
$\rn\times\rn$ and is sufficiently smooth.
The main difference between the aforementioned results of Chaffee et al.
and Theorem \ref{compact-thm} is that the bilinear Riesz transform,
or the Calder\'on--Zygmund operator with kernel of the form \eqref{roughKernel},
does not satisfy \eqref{decay} and, conversely, the operator $T$
in Theorem \ref{compact-thm} surely does not have the form \eqref{roughKernel}.
Thus, the operators considering, respectively, in aforementioned results of
Chaffee et al. and Theorem \ref{compact-thm} are two completely different classes
of operators, and hence the corresponding theorems are also completely unrelated.

Besides, it is still an challenging open problem to find a class of
bilinear Calder\'on--Zygmund operators $\widetilde{T}$,
whose kernels satisfy \eqref{sizeregular} and \eqref{decay},
such that $\{[b,\widetilde{T}]_i\}_{i=1}^2$ are compact from
$L^{p_1}(\rn)\times L^{p_2}(\rn)$ to $L^{p}(\rn)$
if and only if $b\in\XMO(\rn)$,
where $p_1,\ p_2\in(1,\fz)$ and $p\in(\frac12,\fz)$
satisfy $\frac1p=\frac1{p_1}+\frac1{p_2}$.
\end{itemize}
\end{remark}

The remainder of this article is organized as follows.

In Section \ref{section2}, we first notice the nontriviality of $\XMO(\rn)$ when $n=1$,
namely,
$$\XMO(\rr)\subsetneqq\VMO(\rr);$$
see Proposition \ref{xMO} below.
Based on its calculation, we further show that $\XMO(\rn)$ has a similar
equivalent characterization as $\VMO(\rn)$ and $\CMO(\rn)$;
see Theorem \ref{xMO-char} below.
To achieve this, geometrically inspired by Uchiyama \cite{Uchiyama78TohokuMathJ},
we first approximate a given function
$f\in\XMO(\rn)$ by an exceptional simple function $g_\ez$ which is constructed
based on a dyadic family $\cf$, and some essential new techniques on dyadic cubes.
These new techniques provide some exponential decay property of the mean oscillation $\co(f,Q)$
when $Q$ is far away from the origin.
Roughly speaking, $\cf$ consists of numerous small equal-size dyadic cubes
near the origin, and farther away from the origin,
the larger the dyadic cubes in $\cf$ are.
Moreover, by the convolution of $g_\ez$ and
an even function $\varphi$ with delicate dilation
which strongly depends on $\ez$ and some
exquisite geometrical observations of $\cf$,
we construct an approximation element $h_\ez$ of $f$ in the $\BMO(\rn)$ norm.
To prove $h_\ez\in B_\fz(\rn)$,
we use a key analytic technic, namely, first to prove
$\lim_{|x|\to\fz}D^\az h_\ez(x)=0$ whenever $|\az|$ is odd
via the aforementioned exponential decay property;
from this and the Taylor remainder theorem,
we then deduce $\lim_{|x|\to\fz}D^\az h_\ez(x)=0$ whenever $|\az|$ is even,
which further implies that $h_\ez\in B_\fz(\rn)$
and finally completes the proof of Theorem \ref{xMO-char}.
As a corollary, we obtain
$$\xMO(\rn)=\XMO(\rn)\subsetneqq\VMO(\rn)\,\,$$
in Corollary \ref{xMO-coro} below,
which completely answers the open question raised in \cite{TX19}.

In Section \ref{section3}, we give the proof of Theorem \ref{compact-thm}.
Since a general $A_p$ weight is not invariant under translations,
the method in \cite{TX19} can not be applied to the weighted setting directly.
Thus, to overcome this difficulty,
a main new idea is to change the dominations of the
translation-invariant positive operators in \cite{TX19}
into the dominations of the maximal functions and the
smooth truncated Calder\'on--Zygmund operators.
To this end, we use several smooth truncated techniques and
the density arguments of compact operators.
Especially, using this method,
we can also optimize \cite[Theorem 1.1]{TX19}
from ``$K$ satisfies \eqref{N=123}'' to
``$K$ satisfies \eqref{decay}'' even in the unweighted case.

Throughout this article, we denote by $C$ and $\widetilde{C}$
{positive constants} which are independent of main parameters,
but they may vary from line to line.
Moreover, we use $C_{(\gamma,\ \beta,\ \ldots)}$ to denote
a positive constant depending on the indicated
parameters $\gamma,\ \beta,\ \ldots$.
Constants with subscripts, such as $C_{0}$ and $A_1$,
do not change in different occurrences.
Moreover, the {symbol} $f\lesssim g$ represents that
$f\le Cg$ for some positive constant $C$.
If $f\lesssim g$ and $g\lesssim f$, we then write $f\sim g$.
If $f\le Cg$ and $g=h$ or $g\le h$, we then write $f\ls g\sim h$
or $f\ls g\ls h$, \emph{rather than} $f\ls g=h$ or $f\ls g\le h$.
Let $\mathbb{N}:=\{1,\,2,...\}$ and $\mathbb{Z}_+:=\mathbb{N}\cup\{0\}$.
For any $p\in[1,\fz]$, let $p'$ denote its \emph{conjugate index},
that is, $p'$ satisfies  $1/p+1/p'=1$.
For any cube $Q\subsetneqq\rn$ and $f\in L_{\loc}^1(\rn)$, let
$$\fint_Q:=\frac1{|Q|}\int_Q\quad{\rm and}\quad f_Q:=\fint_Q f(y)\,dy;$$
moreover, the \emph{mean oscillation} $\co(f;Q)$ is defined by setting
$$\co(f;Q):=\fint_Q\lf|f(x)-f_Q\r|\,dx.$$

\section{Characterization and non-triviality of $\XMO(\rn)$}\label{section2}

In this section, we investigate the equivalent characterization of $\XMO(\rn)$.
To this end, we begin with the following concise counterexample on the real line.
\begin{proposition}\label{xMO}
There exists some $f\in \VMO(\rr)\setminus\XMO(\rr)$.
\end{proposition}

\begin{proof}
For any $x\in\rr$, let $f(x):=\sin(x)$. Then $f$ is
uniformly continuous and $f\in L^\fz(\rr)\subset\BMO(\rr)$.
Thus, $f\in \VMO(\rr)$. We claim that, for any $g\in B_1(\rn)$,
\begin{align}\label{d(f,B)}
\|f-g\|_{\BMO(\rr)}\ge\frac1{2\pi}.
\end{align}
Indeed, for any $k\in\nn$, let
$$I_k:=\lf[2k\pi-\frac\pi2,2k\pi+\frac\pi2\r].$$
Since $g\in B_1(\rn)$, it follows that $\lim_{|x|\to\fz}g'(x)=0$ and hence
we can choose $k$ large enough such that, for any $y\in I_k$,
\begin{equation}\label{g' small}
|g'(y)|<\frac4{\pi^2}<\frac{\sqrt2}{\pi}.
\end{equation}
Therefore, by the mean value theorem and the fundamental theorem of calculus,
we have
\begin{align}\label{eq1}
\mathcal{O}(f-g;I_k)=&\frac{1}{|I_k|}\int_{I_k}\lf|(f-g)(x)-(f-g)_{I_k}\r|\,dx\\
=&\frac{1}{|I_k|}\int_{I_k}\lf|(f-g)(x)-(f-g)(\xi_k)\r|\,dx\notag
=\frac{1}{|I_k|}\int_{I_k}\lf|\int_{\xi_k}^x(f-g)'(y)\,dy\r|\,dx\notag\\
=&\frac{1}{|I_k|}\int_{I_k}\lf|\int_{\xi_k}^x\lf[\cos(y)-g'(y)\r]\,dy\r|\,dx\notag,
\end{align}
where $\xi_k\in I_k$ is independent of $x$, but it may depend on $k$.
Without loss of generality, we may assume that $\xi_k\in[2k\pi-\frac\pi2,2k\pi]$.
Then, from \eqref{eq1}, we deduce that
\begin{align}\label{eq2}
\mathcal{O}(f-g;I_k)\ge\frac1\pi\int_{2k\pi}^{2k\pi+\frac\pi2}
     \lf|\int_{\xi_k}^x\lf[\cos(y)-g'(y)\r]\,dy\r|\,dx.
\end{align}
By the fact that $\xi_k\in[2k\pi-\frac\pi2,2k\pi]$
and $x\in[2k\pi,2k\pi+\frac\pi2]$,
we know that $\frac{x+\xi_k}{2}\in[2k\pi-\frac\pi4,2k\pi+\frac\pi4]$ and
$\frac{x-\xi_k}{2}\in[0,\frac\pi2]$. Therefore,
\begin{align}\label{sin cos}
\cos\lf(\frac{x+\xi_k}{2}\r)\ge \frac{\sqrt2}2\quad\mathrm{and}\quad
    \sin\lf(\frac{x-\xi_k}{2}\r)\ge\frac2\pi\frac{x-\xi_k}2=\frac{x-\xi_k}{\pi}.
\end{align}
From \eqref{g' small} and \eqref{sin cos}, it follows that
\begin{align}\label{positive}
\int_{\xi_k}^x\lf[\cos(y)-g'(y)\r]\,dy
&>\int_{\xi_k}^x\lf[\cos(y)-\frac{\sqrt2}\pi\r]\,dy
=\sin(x)-\sin(\xi_k)-\frac{\sqrt2}\pi(x-\xi_k)\\
&=2\cos\lf(\frac{x+\xi_k}{2}\r)\sin\lf(\frac{x-\xi_k}{2}\r)
  -\frac{\sqrt2}\pi(x-\xi_k)\notag\\
&\ge2\frac{\sqrt2}2\frac{x-\xi_k}{\pi}-\frac{\sqrt2}\pi(x-\xi_k)
 \ge0\notag.
\end{align}
By \eqref{eq2}, \eqref{positive}, $\xi_k\in[2k\pi-\frac\pi2,2k\pi]$
and \eqref{g' small}, we conclude that
\begin{align*}
\mathcal{O}(f-g;I_k)&\ge \frac1\pi\int_{2k\pi}^{2k\pi+\frac\pi2}
\int_{\xi_k}^x\lf[\cos(y)-g'(y)\r]\,dy\,dx\ge \frac1\pi\int_{2k\pi}^{2k\pi+\frac\pi2}
\int_{2k\pi}^x\lf[\cos(y)-\frac4{\pi^2}\r]\,dy\,dx\\
&=\frac1\pi\int_{2k\pi}^{2k\pi+\frac\pi2}
     \lf[\sin(x)-\frac4{\pi^2}(x-2k\pi)\r]\,dx\\
&=\frac1\pi\lf(1-\frac4{\pi^2}\int_0^\frac\pi2z\,dz\r)
=\frac1\pi\lf[1-\frac4{\pi^2}\frac12\lf(\frac{\pi}2\r)^2 \r]=\frac1{2\pi}.
\end{align*}
This implies that the inequality $\eqref{d(f,B)}$ holds true,
which completes the proof of Proposition \ref{xMO}.
\end{proof}

\begin{remark}
One can modify the above calculation from $\rr$ to $\rn$,
but this process may be tedious.
However, if, for any $(x_1,...,x_n)\in\rn$, let
$$f(x_1,...,x_n):=\prod_{k=1}^n\sin(x_k)$$
then, by Theorem \ref{xMO-char}, we immediately know that
$$f\in \VMO(\rn)\setminus\XMO(\rn);$$
see the proof of Corollary \ref{xMO-coro} below.
\end{remark}

In what follows, we need to use the following equivalent
characterizations of $\VMO(\rn)$  and $\CMO(\rn)$ established by
Sarason \cite{Sarason75} and Uchiyama \cite{Uchiyama78TohokuMathJ},
respectively.

\begin{proposition}(\cite[Theorem 1]{Sarason75})\label{VMO-char}
Let $f\in\BMO(\rn)$.
Then $f\in\VMO(\rn)$ if and only if
$$\lim_{a\to0^+}\sup_{|Q|=a}\co(f; Q)=0.$$
\end{proposition}

\begin{proposition}(\cite[p.\,166]{Uchiyama78TohokuMathJ})\label{CMO-char}
Let $f\in\BMO(\rn)$.
Then $f\in\CMO(\rn)$ if and only if $f$ satisfies the following three conditions:
\begin{itemize}
\item [{\rm(i)}] $$\lim_{a\to0^+}\sup_{|Q|=a}\co(f; Q)=0;$$
\item [{\rm (ii)}]for any cube $Q\subset\rn$, $$\lim_{|x|\to\fz}\co(f; Q+x)=0;$$
\item [{\rm(iii)}]$$\lim_{a\to\fz}\sup_{|Q|=a}\co(f; Q)=0.$$
\end{itemize}
\end{proposition}

Observe that, in the proof of Proposition \ref{xMO},
the mean oscillations $\{\co(f;I_k)\}_{k\in\nn}$ violate
Proposition \ref{CMO-char}(ii), which leads us to consider
the limit condition (ii)$_2$ of Theorem \ref{xMO-char}(ii).

Now, we are in the position to give the proof of Theorem \ref{xMO-char}.

\begin{proof}[Proof of Theorem \ref{xMO-char}]
We first prove (i) $\Longrightarrow$ (ii).
By the density argument, it suffices to show that,
for any $f\in B_1(\rn)$, both (ii)$_1$ and (ii)$_2$
of Theorem \ref{xMO-char}(ii) hold true.
Indeed, for any $x,\ y\in\rn$, by the mean value theorem, we obtain
$$|f(x)-f(y)|=|\nabla f(\xi)\cdot(x-y)|\le\|\nabla f\|_{L^\fz(\rn)}|x-y|,$$
where $\xi$ is on the segment $\overline{xy}$ connecting $x$ and $y$, and
$\|\nabla f\|_{L^\fz(\rn)}<\fz$ because $f\in B_1(\rn)$.
This implies that $f\in C_{\rm u}(\rn)$ and hence $f\in \VMO(\rn)$.
From this and Proposition \ref{VMO-char},
it follows that $f$ satisfies (ii)$_1$ of Theorem \ref{xMO-char}(ii).
Moreover, for any fixed cube $Q\subsetneqq\rn$, by the mean value theorem again,
we conclude that
\begin{align*}
\co(f;Q)&=\frac1{|Q|}\int_Q\lf|f(x)-\frac1{|Q|}\int_Q f(y)\,dy\r|\,dx
\le\frac1{|Q|^2}\int_Q\int_Q|f(x)-f(y)|\,dy\,dx\\
&=\frac1{|Q|^2}\int_Q\int_Q|\nabla f(\xi)\cdot(x-y)|\,dy\,dx
\ls\frac1{|Q|^2}\int_Q\int_Q\lf|\nabla f(\xi)\r||Q|^\frac1n\,dy\,dx\\
&\ls \lf[\sup_{z\in Q}|\nabla f(z)|\r]|Q|^\frac1n.
\end{align*}
Thus, for the given cube $Q$ and any $x\in\rn$, we have
$$\co(f;Q+x)\ls \lf[\sup_{z\in Q+x}|\nabla f(z)|\r]|Q+x|^\frac1n
\sim \lf[\sup_{z\in Q+x}|\nabla f(z)|\r]|Q|^\frac1n\to0$$
as $|x|\to\fz$, which shows that
$f$ satisfies (ii)$_2$ of Theorem \ref{xMO-char}(ii).
This finishes the proof that (i) $\Longrightarrow$ (ii).

Now, we prove (ii) $\Longrightarrow$ (iii).
Let $f\in\BMO(\rn)$ satisfy both (ii)$_1$ and (ii)$_2$
of Theorem \ref{xMO-char}(ii).
To prove $f\in\XMO(\rn)$, for any fixed $\ez\in(0,\fz)$,
it suffices to show that there exist a simple function
$g_\ez$ satisfying
\begin{align}\label{f-g_e}
\|f-g_\ez\|_{\BMO(\rn)}\ls \ez,
\end{align}
and a function $h_\ez\in B_\fz(\rn)$ satisfying
\begin{align}\label{g_e-h_e}
\|g_\ez-h_\ez\|_{\BMO(\rn)}\ls \ez.
\end{align}

The remainder of the proof that (ii) $\Longrightarrow$ (iii)
consists of the following three steps:
\begin{enumerate}
\item [{\textbf{Step i)}}] Construct a family $\cf$ of disjoint dyadic cubes
and introduce a simple function $g_\ez$ via $\cf$.
\item [{\textbf{Step ii)}}] Show that \eqref{f-g_e} holds true.
\item [{\textbf{Step iii)}}] Define $h_\ez$ via $g_\ez$, and then show that
  \eqref{g_e-h_e} holds true and $h_\ez\in B_\fz(\rn)$.
\end{enumerate}

We proceed in order and begin with \textbf{Step i)}.
For the above given $\ez\in(0,\fz)$, by (ii)$_1$ of Theorem \ref{xMO-char}(ii),
we know that
there exists a negative integer $j(\ez;0)\in\zz_-:=\{-1,-2,\dots\}$ such that,
for any cube $Q$ with the side length $\ell(Q)<2^{j(\ez;0)+1}$,
\begin{align}\label{j(e;0)}
\co(f;Q)<\ez.
\end{align}
Here and thereafter, we denote the side length of a cube $Q$ by $\ell(Q)$.
Besides, we always use $Q(x,r)$ to denote the cube
centered at $x$ with the side length $2r$,
and $\cd$ to denote the family of all classical dyadic cubes in $\rn$.
By (ii)$_2$ of Theorem \ref{xMO-char}(ii), we find that
there exists some $j(\ez;1)\in\zz$ with $j(\ez;1)>j(\ez;0)$
such that, for any $x\in\rn$ with $|x|\ge j(\ez;1)$,
\begin{align}\label{j(e;1)}
\co(f;Q(x,2^{j(\ez;0)+1}))<2^{j(\ez;0)}\ez
  \le 2^{-1}\ez<\ez.
\end{align}
Repeating the above procedure, we obtain, for any $k\in\nn$,
there exists some $j(\ez;k)\in\zz$
with $j(\ez;k)>j(\ez;k-1)>\cdots>j(\ez;0)$ such that,
for any $x\in\rn$ with $|x|\ge 2^{j(\ez;k)}$,
\begin{align}\label{j(e;k)}
\co(f;Q(x,2^{j(\ez;0)+k}))<2^{kj(\ez;0)}\ez<\ez.
\end{align}
Now, define $\{\cf_k\}_{k\in\nn}$ and $\cf$ as follows:
\begin{align*}
\cf_1:=&\lf\{Q\subset \overline{Q(\vec 0_n,2^{j(\ez;1)})}:
\,\,Q\in\cd \mathrm{\,\,with\,\,}\ell(Q)=2^{j(\ez;0)} \r\};\\
\cf_2:=&\lf\{Q\subset \overline{Q(\vec 0_n,2^{j(\ez;2)})
     \setminus Q(\vec 0_n,2^{j(\ez;1)})}:
\,\,Q\in\cd \mathrm{\,\,with\,\,}\ell(Q)=2^{j(\ez;0)+1} \r\};\\
\vdots\\
\cf_k:=&\lf\{Q\subset \overline{Q(\vec 0_n,2^{j(\ez;k)})
     \setminus Q(\vec 0_n,2^{j(\ez;k-1)})}:
\,\,Q\in\cd \mathrm{\,\,with\,\,}\ell(Q)=2^{j(\ez;0)+k-1} \r\};\\
\vdots
\end{align*}
and
$$\cf:=\bigcup_{k\in\nn}\cf_k,$$
here and thereafter, for any subset $A$ of $\rn$,
we use $\overline A$ to denote its closure in $\rn$.
Then, for any $k\in\nn$, $\cf_k$ contains disjoint cubes
with the same side length and hence $\cf$ is a family of disjoint dyadic cubes.
Next, we introduce the simple function $g_\ez$ associated with $\cf$ as follows.
Since the cubes in $\cf$ are disjoint, it follows that, for any $x\in\rn$,
there exists a unique cube $Q_{(x)}\in\cf$ such that $Q_{(x)}\ni x$; let
\begin{align}\label{g_e-def}
g_\ez(x):=f_{Q_{(x)}}:=\frac{1}{|Q_{(x)}|}\int_{Q_{(x)}}f(y)\,dy.
\end{align}
Then $g_\ez$ is a simple function on $\rn$.
This finishes the proof of \textbf{Step i)}.

\textbf{Step ii)} To estimate $\|f-g_\ez\|_{\BMO(\rn)}$,
we first claim that, for any $x,\ y\in\rn$ with
$\overline{Q_{(x)}}\cap \overline{Q_{(y)}}\neq\emptyset$,
\begin{align}\label{g(x)-g(y)}
|g_\ez(x)-g_\ez(y)|\ls\ez.
\end{align}
Indeed, if both $x$ and $y$ lie in the same cube $Q\in\cf$, then,
by the definition of $g_\ez$, we know that $g_\ez(x)=g_\ez(y)$
and hence \eqref{g(x)-g(y)} holds true trivially.
If $x$ and $y$ lie, respectively, in different dyadic cubes $Q_{(x)}$ and  $Q_{(y)}$,
then, from the construction of $\cf$,
it follows that $Q_{(x)}$ and  $Q_{(y)}$ must be adjacent,
namely, $\overline{Q_{(x)}}\cap \overline{Q_{(y)}}$ is a point,
segment or surface. Anyhow, $|Q_{(x)}|$ and $|Q_{(y)}|$ are comparable
and hence there exists a larger dyadic cube $Q_{(x,y)}\in\cd$ such that
$$Q_{(x)}\subset Q_{(x,y)}\quad\mathrm{and}\quad Q_{(y)}\subset Q_{(x,y)},$$
whose side length
$$\ell(Q_{(x,y)})=2\max\{\ell(Q_{(x)}),\ell(Q_{(y)})\}
\sim\ell(Q_{(x)})\sim\ell(Q_{(y)}),$$
where $\ell(Q_{(x)})$ and $\ell(Q_{(y)})$ denote the side lengths of
$Q_{(x)}$ and $Q_{(y)}$, respectively.
From the definition of $\cf$ and \eqref{j(e;k)}, we deduce that
$$\co(f;Q_{(x,y)})<\ez$$
and hence
\begin{align}\label{k(x,y)}
|g_\ez(x)-g_\ez(y)|
&\le\lf|f_{Q_{(x)}}-f_{Q_{(x,y)}}\r|+\lf|f_{Q_{(y)}}-f_{Q_{(x,y)}}\r|
\le2\lf[\frac{|Q_{(x,y)}|}{|Q_{(x)}|}+\frac{|Q_{(x,y)}|}{|Q_{(y)}|}\r]\co(f;Q_{(x,y)})\\
&\ls\co(f;Q_{(x,y)})\ls\ez.\notag
\end{align}
Thus, \eqref{g(x)-g(y)} also holds true in this case.
This finishes the proof of the above claim.

Now, we estimate $\|f-g_\ez\|_{\BMO(\rn)}:=\sup_{Q}\co(f-g_\ez;Q)$
via considering different side lengths $\ell(Q)$ in the supremum.

When $\ell(Q)\in(0,2^{j(\ez;0)})$,
by the definition of $\cf$, $Q$ intersects at most $2^n$ different cubes in $\cf$.
From this, the definition of $g_\ez$ and \eqref{g(x)-g(y)}, we deduce that
\begin{align}\label{less2^n}
\frac1{|Q|^2}\int_Q\int_Q|g_\ez(x)-g_\ez(y)|\,dxdy\ls\ez.
\end{align}
Combining \eqref{j(e;0)} with \eqref{less2^n}, we obtain
\begin{align*}
\co(f-g_\ez;Q)&\le\co(f;Q)+\co(g_\ez;Q)
<\ez+\frac1{|Q|^2}\int_Q\int_Q|g_\ez(x)-g_\ez(y)|\,dxdy\ls\ez.
\end{align*}

When $\ell(Q)\in[2^{j(\ez;0)},2^{j(\ez;0)+1})$,
we consider the following two cases:

{\textbf{Case i)}} $Q\cap Q(\vec0_n,2^{j(\ez;1)})= Q$, namely,
$Q\subset Q(\vec0_n,2^{j(\ez;1)})$.
In this case, by the definition of $\cf_1$,
$Q$ intersects at most $3^n$ different cubes in $\cf_1$.
This, together with the definition of $g_\ez$ and \eqref{j(e;0)}, implies that
\begin{align*}
\co(f-g_\ez;Q)&\le\frac2{|Q|}\int_Q \lf|f(x)-g_\ez(x)\r|\,dx
 =\sum_{Q_{\ast}\in\cf_1:\,\,Q\cap Q_{\ast}\neq\emptyset}\frac2{|Q|}
  \int_{Q_{\ast}} \lf|f(x)-g_\ez(x)\r|\,dx\\
&=2\sum_{Q_{\ast}\in\cf_1:\,\,Q\cap Q_{\ast}\neq\emptyset}
  \frac{|Q_{\ast}|}{|Q|}\fint_{Q_{\ast}}\lf|f(x)-f_{Q_\ast}\r|\,dx\\
&<2\ez\sum_{Q_{\ast}\in\cf_1:\,\,Q\cap Q_{\ast}\neq\emptyset}
  \frac{|Q_{\ast}|}{|Q|}\ls\ez.
\end{align*}

{\textbf{Case ii)}} $Q\cap Q(\vec0_n,2^{j(\ez;1)})\neq Q$.
In this case, we claim that there exists some $x_Q\in\rn$ such that
\begin{align}\label{x_Q}
Q\subset Q(x_Q,2^{j(\ez;0)+1})\quad {\rm and}\quad |x_Q|>2^{j(\ez;1)}.
\end{align}
Indeed, if $Q\cap Q(\vec0_n,2^{j(\ez;1)})=\emptyset$,
we can apparently choose $x_Q$ to be the center of $Q$
and, if $Q\cap Q(\vec0_n,2^{j(\ez;1)})\neq\emptyset$,
the existence of $x_Q$ is obtained from the fact that
the distance between the center of $Q$ and the boundary of
$Q(\vec0_n,2^{j(\ez;1)})$ is less than $\frac12\ell(Q)<2^{j(\ez;0)}$.
Thus, the above claim holds true.
By \eqref{x_Q} and \eqref{j(e;1)}, we find that
\begin{align}\label{O(f;Q)}
\co(f;Q)\le2\frac{[2^{j(\ez;0)+2}]^n}{|Q|}\co(f,Q(x_Q;2^{j(\ez;0)+1}))
\ls\ez.
\end{align}
Meanwhile, by the definition of $\cf$, $Q$ intersects at most
$3^n$ different cubes in $\cf$.
Therefore, \eqref{less2^n} still holds true.
Combining \eqref{O(f;Q)} and \eqref{less2^n}, we obtain
\begin{align*}
\co(f-g_\ez;Q)&\le\co(f;Q)+\co(g_\ez;Q)
\ls\ez+\frac1{|Q|^2}\int_Q\int_Q|g_\ez(x)-g_\ez(y)|\,dxdy\ls\ez.
\end{align*}

Combining \textbf{Case i)} and \textbf{Case ii)},
we finally conclude that $\co(f-g_\ez;Q)\ls \ez$
when $\ell(Q)\in[2^{j(\ez;0)},2^{j(\ez;0)+1})$.

Observe that, by the  geometrical property of $\cf$,
for any $k\in\nn$, the above estimations when $\ell(Q)\in[2^{j(\ez;0)},2^{j(\ez;0)+1})$
can be modified into the case $\ell(Q)\in[2^{j(\ez;0)+k-1},2^{j(\ez;0)+k})$
with the implicit positive constant depending only on the dimension $n$.
This finishes the proof of \textbf{Step ii)}.

\textbf{Step iii)}
Let $\varphi\in C_{\rm c}^\fz(\rn)$ be a non-negative even function with
$\int_{\rn}\varphi(x)\,dx=1$ and
$$\supp(\varphi)\subset B(\vec0_n,1):=\{x\in\rn:\,\,|x|\le1\}.$$
Let $h_\ez:=g_\ez\ast\varphi_{2^{j(\ez;0)}}$, where
$\varphi_{2^{j(\ez;0)}}(\cdot)
:=2^{-nj(\ez;0)}\varphi(2^{-j(\ez;0)}\cdot)$.
Notice that, for any $x,\ y\in\rn$ with $|x-y|\le 2^{j(\ez;0)}$,
by the definition of $\cf$, we know that
$\overline{Q_{(x)}}\cap \overline{Q_{(y)}}\neq\emptyset$.
Then, for any $x\in\rn$, by \eqref{g(x)-g(y)}, we have
\begin{align*}
\lf|g_\ez(x)-h_\ez(x)\r|
&=\lf|\int_{\rn}[g_\ez(x)-g_\ez(y)]\varphi_{2^{j(\ez;0)}}(x-y)\,dy\r|\\
&\le\int_{B(x,2^{j(\ez;0)})}|g_\ez(x)-g_\ez(y)|
   |\varphi_{2^{j(\ez;0)}}(x-y)|\,dy\\
&\ls\ez\int_{B(x,2^{j(\ez;0)})}|\varphi_{2^{j(\ez;0)}}(x-y)|\,dy
\sim\ez,
\end{align*}
where $B(x,2^{j(\ez;0)})$ denotes the ball centered at $x$
with radius $2^{j(\ez;0)}$. Thus,
$$\lf\|g_\ez-h_\ez\r\|_{\BMO(\rn)}
\le2\lf\|g_\ez-h_\ez\r\|_{L^\fz(\rn)}\ls\ez,$$
which shows that \eqref{g_e-h_e} holds true.

It remains to prove that $h_\ez\in B_\fz(\rn)$.
Indeed, by $\varphi\in C_{\rm c}^\fz(\rn)$, \eqref{f-g_e},
\eqref{g_e-h_e} and $f\in\BMO(\rn)$,
we know that $h_\ez\in C^\fz(\rn)$ and $h_\ez\in \BMO(\rn)$.
Thus, to show $h_\ez\in B_\fz(\rn)$, it suffices to prove that,
for any given $\widetilde{\ez}\in(0,\ez)$ and $\az\in\zz_+^n\setminus\{\vec 0_n\}$,
any $x\in\rn$ and $|x|>E_{(\az,n)}$ with $E_{(\az,n)}\in (0,\fz)$ being determined later,
$$\lf|D^\az h_\ez(x)\r|\ls\widetilde{\ez}.$$

Indeed, when  $\az\in\zz_+^n$ and $|\az|\in\{2m-1\}_{m\in\nn}$,
by $j(\ez;0)<0$ and $\widetilde{\ez}<\ez$,
we can choose $k_{|\az|}$ to be the smallest positive integer such that
\begin{align}\label{eez}
2^{(|\az|+k_{|\az|})j(\ez;0)}\ez\le 2^{|\az|j(\ez;0)}\widetilde{\ez}
\end{align}
and $\{k_{|\az|}\}_{|\az|\in\{2m-1\}_{m\in\nn}}$ is increasing, namely,
\begin{align}\label{kaz}
k_1\le\cdots\le k_{|\az|}\le k_{|\az|+2}\le\cdots.
\end{align}
Meanwhile, from the fact that $\varphi$ is even,
we deduce that $D^\az\varphi$ is odd and hence
\begin{align}\label{int=0}
\int_{\rn}D^\az\varphi(x)\,dx=0.
\end{align}
Also, in this case, for any $x\in\rn$ with
$|x|>E_{(\az,n)}:=\sqrt n2^{j(\ez;|\az|+k_{|\az|})}$
and any $y\in B(x,2^{j(\ez;0)})$, by the definition of $\cf$, we have
$\overline{Q_{(x)}}\cap Q(\vec 0_n,2^{j(\ez;|\az|+k_{|\az|})})=\emptyset$
and
$\overline{Q_{(x)}}\cap \overline{Q_{(y)}}\neq\emptyset$,
which, combined with \eqref{k(x,y)}, the definition of $\cf$ and \eqref{j(e;k)}, further implies that
\begin{align}\label{k(x,y)-sharp}
\lf|g_\ez(y)-g_\ez(x)\r|\ls\co(f;Q_{(x,y)})
\ls 2^{(|\az|+k_{|\az|})j(\ez;0)}\ez
\ls 2^{|\az|j(\ez;0)}\widetilde{\ez},
\end{align}
where $Q_{(x,y)}\supset(Q_{(x)}\cup Q_{(y)})$ is the dyadic cube
comparable with both $Q_{(x)}$ and $Q_{(y)}$
[see the first paragraph of the proof of \textbf{Step ii)} above] and the implicit
positive constant only depends on $n$.
By \eqref{int=0} and \eqref{k(x,y)-sharp}, we conclude that, for any
$\az\in\zz_+^n$ with $|\az|\in\{2m-1\}_{m\in\nn}$, and any $x\in\rn$
with $|x|>E_{(\az,n)}$,
\begin{align}\label{odd}
|D^\az h_\ez(x)|
&=\lf|\int_{\rn}g_\ez(y) D^\az \varphi_{2^{j(\ez;0)}}(x-y)\,dy\r|
 =\lf|\int_{B(x,2^{j(\ez;0)})}g_\ez(y)
  D^\az \varphi_{2^{j(\ez;0)}}(x-y)\,dy\r|\\
&=\lf|\int_{B(x,2^{j(\ez;0)})}[g_\ez(y)-f_{Q_{(x)}}]
   D^\az \varphi_{2^{j(\ez;0)}}(x-y)\,dy\r|\notag\\
&=\lf|\int_{B(x,2^{j(\ez;0)})}[g_\ez(y)-g_\ez(x)]
   D^\az \varphi_{2^{j(\ez;0)}}(x-y)\,dy\r|\notag\\
&\ls \int_{B(x,2^{j(\ez;0)})}\co(f;Q_{(x,y)})
  |D^\az \varphi_{2^{j(\ez;0)}}(x-y)|\,dy\notag\\
&\ls 2^{|\az|j(\ez;0)}\widetilde{\ez}
  \int_{B(x,2^{j(\ez;0)})}|D^\az \varphi_{2^{j(\ez;0)}}(x-y)|\,dy\notag\\
&\ls 2^{|\az|j(\ez;0)}\widetilde{\ez} 2^{-|\az|j(\ez;0)}\|D^\az\varphi\|_{L^1(\rn)}
 \ls \widetilde{\ez}, \notag
\end{align}
where the implicit positive constant is independent of $\widetilde{\ez}$ and $x$.

When $\az\in\zz_+^n$ and $|\az|\in\{2m\}_{m\in\nn}$,
we claim that $|D^\az h_\ez(x)|\ls\widetilde{\ez}$ as well
for any given $\az\in\zz_+^n$ with $|\az|\in\{2m\}_{m\in\nn}$ and any
$|x|>E_{(\az,n)}:=\sqrt n2^{j(\ez;|\az|+1+k_{|\az|+1})}$, with the implicit
positive constant independent of $\wz\ez$ and $x$.
Indeed, let $\psi\in C^\fz(\rn)$ and $M$ be a positive constant.
By the Taylor remainder theorem, we conclude that,
for any $x:=(x_1,\ldots,x_n)\in\rn$ with $|x|>M$ and
$y\in \mathcal{R}_x:=\{y:=(y_1,\ldots,y_n)\in\rn:\,\,x_iy_i\ge0\quad
 \forall\,i\in\{1,\ldots,n\}\}$,
\begin{align}\label{taylor}
\psi(x+y)=\psi(x)+\sum_{i=1}^n\frac{\partial}{\partial {x_i}}\psi(x)y_i
  +\sum_{\{\beta\in\zz^n_+:\,\,|\beta|=2\}} R_\beta(x,y) y^\beta,
\end{align}
where, for any $\beta:=(\bz_1,\ldots,\bz_n)\in\zz^n_+$ and $|\beta|=2$,
\begin{align*}
R_\beta(x,y):=\frac{|\beta|}{\beta!}\int_0^1(1-t)^{|\beta|-1}
  D^\beta \psi(x+ty)\,dt,
\end{align*}
with $\beta!:=\beta_1!\cdots\beta_n!$,
satisfies
\begin{align*}
\lf|R_\beta(x,y) \r|\le\max_{\{\beta\in\zz^n_+:\,\,|\beta|=2\}}
  \frac{1}{\beta!}\sup_{|z|>M}\lf|D^\beta\psi(z) \r|
\le\max_{\{\beta\in\zz^n_+:\,\,|\beta|=2\}}\|D^\beta\psi\|_{L^\fz(\{|z|>M\})},
\end{align*}
by using the following observation that, for any $t\in [0,1]$, $|x+ty|\ge |x|>M$.
To estimate $\|\frac{\partial}{\partial x_1}\psi\|_{L^\fz(\{|z|>M\})}$,
let
$$y\in\mathcal{R}_x^{(1)}:=
\lf\{y:=(y_1,\ldots,y_n)\in\rn:\,\,x_1y_1\ge0,\,
y_1\neq0,\, y_i=0\quad \forall\,i\in\{2,\ldots,n\}\r\};$$
then $|x+y|\ge|x|>M$, and \eqref{taylor} becomes
\begin{align*}
\psi(x+y)=\psi(x)+\frac{\partial}{\partial {x_1}}\psi(x)y_1+
  R_{(2,0,\dots,0)}(x,y) y_1^{2},
\end{align*}
which imply that
\begin{align*}
\lf|\frac{\partial}{\partial x_1}\psi(x)\r|&\le|\psi(x+y)-\psi(x)||y_1|^{-1}
+\lf|R_{(2,0,\dots,0)}(x,y)\r||y_1|\\
&\le 2\|\psi\|_{L^\fz(\{|z|>M\})}|y_1|^{-1}
 +\max_{\{\beta\in\zz^n_+:\,\,|\beta|=2\}}\|D^\beta\psi\|_{L^\fz(\{|z|>M\})}|y_1|.
\end{align*}
From the arbitrariness of both $x\in\{z\in\rn:\,\,|z|>M\}$
and $|y_1|\in(0,\fz)$, and
the AM-GM inequality\footnote{the inequality of arithmetic and geometric means, namely,
for any $a$, $b\in[0,\fz)$, $\frac{a+b}2\ge\sqrt{ab}$
and, moreover, the equality holds true when $a=b$.}, we then deduce that
\begin{align}\label{beta=2}
\lf\|\frac{\partial}{\partial x_1}\psi\r\|_{L^\fz(\{|z|>M\})}
&\le \inf_{|y_1|>0}\lf[2\|\psi\|_{L^\fz(\{|z|>M\})}|y_1|^{-1}
  +\max_{\{\beta\in\zz^n_+:\,\,|\beta|=2\}}\|D^\beta\psi\|_{L^\fz(\{|z|>M\})}|y_1|\r]\\
&= \sqrt{2\|\psi\|_{L^\fz(\{|z|>M\})}
  \max_{\{\beta\in\zz^n_+:\,\,|\beta|=2\}}\|D^\beta\psi\|_{L^\fz(\{|z|>M\})}}\notag.
\end{align}
By the same technique, we know that \eqref{beta=2} also holds true with $\frac{\partial}{\partial x_1}\psi$
replaced by $\frac{\partial}{\partial x_i}\psi$ for any $i\in\{2,\ldots,n\}$.
Based on this, we can now estimate $\|D^\az h_\ez\|_{L^\fz(\{|x|>E_{(\az,n)}\})}$
for any given $\az:=(\az_1,\ldots,\az_n)\in\zz_+^n$ with
$|\az|\in\{2m\}_{m\in\nn}$.
Without loss of generality, we may assume that $\az_1\neq0$,
and let $\widetilde{\az}:=(\az_1-1,\az_2,\ldots,\az_n)$.
Applying \eqref{beta=2} with $\psi:=D^{\widetilde{\az}}h_\ez$
and $M:=E_{(\az,n)}=\sqrt n2^{j(\ez;|\az|+1+k_{|\az|+1})}$,
we have
\begin{align*}
\lf\|D^\az h_\ez\r\|_{L^\fz(\{|x|>E_{(\az,n)}\})}
&=\lf\|\frac{\partial}{\partial x_1}D^{\widetilde{\az}} h_\ez\r\|_{L^\fz(\{|x|>E_{(\az,n)}\})} \\
&\le \sqrt{2\|D^{\widetilde{\az}}h_\ez\|_{L^\fz(\{|x|>E_{(\az,n)}\})}
    \max_{\{\beta\in\zz^n_+:\,\,|\beta|=|\az|+1\}}
    \|D^\beta h_\ez\|_{L^\fz(\{|x|>E_{(\az,n)}\})}}\\
&= \sqrt{2\|D^{\widetilde{\az}}h_\ez\|_{L^\fz(\{|x|>\sqrt n2^{j(\ez;|\az|+1+k_{|\az|+1})}\})}
    \max_{\{\beta\in\zz^n_+:\,\,|\beta|=|\az|+1\}}
    \|D^\beta h_\ez\|_{L^\fz(\{|x|>\sqrt n2^{j(\ez;|\az|+1+k_{|\az|+1})}\})}}\\
&\le \sqrt{2\|D^{\widetilde{\az}}h_\ez\|_{L^\fz(\{|x|>\sqrt n2^{j(\ez;|\az|-1+k_{|\az|-1})}\})}
    \max_{\{\beta\in\zz^n_+:\,\,|\beta|=|\az|+1\}}
    \|D^\beta h_\ez\|_{L^\fz(\{|x|>\sqrt n2^{j(\ez;|\az|+1)}\})}}\\
&= \sqrt{2\|D^{\widetilde{\az}}h_\ez\|_{L^\fz(\{|x|>E_{(\widetilde{\az},n)}\})}
    \max_{\{\beta\in\zz^n_+:\,\,|\beta|=|\az|+1\}}
    \|D^\beta h_\ez\|_{L^\fz(\{|x|>E_{(\bz,n)}\})}},
\end{align*}
where we used \eqref{kaz} in the last inequality.
From this and \eqref{odd}, we deduce that,
for any $\az:=(\az_1,\ldots,\az_n)\in\zz_+^n$ with $|\az|\in\{2m\}_{m\in\nn}$,
\begin{align}\label{even}
\lf\|D^\az h_\ez\r\|_{L^\fz(\{|x|>E_{(\az,n)}\})}\ls\widetilde{\ez},
\end{align}
where the implicit positive constant is independent of $\ez$.
This finishes the proof of the above claim.

Combining \eqref{odd} and \eqref{even}, we conclude that
$h_\ez\in B_\fz(\rn)$, which completes the proof of \textbf{Step iii)}
and hence that (ii) $\Longrightarrow$ (iii).

The proof that (iii) $\Longrightarrow$ (i) is obvious
from the definitions of $\XMO(\rn)$ and $\xMO(\rn)$.

This finishes the proof of Theorem \ref{xMO-char}.
\end{proof}

By Theorem \ref{xMO-char}, we can now completely answer
the open question posed in \cite{TX19}.

\begin{proof}[Proof of Corollary \ref{xMO-coro}]
To show this corollary,
it suffices to prove that there exists some $f\in\VMO(\rn)\setminus\XMO(\rn)$.
For any $x:=(x_1,\dots,x_n)\in\rn$, define
$$f(x_1,...,x_n):=\prod_{k=1}^n\sin(x_k).$$
Then it is easy to show that $f$ is uniform continuous and bounded in $\rn$,
which implies that $f\in\VMO(\rn)$.

Now, we claim that $f$ violates (ii)$_2$ of Theorem \ref{xMO-char}(ii) and hence
$f\notin \XMO(\rn)$.
Indeed, let $Q_0:=[-\frac\pi2,\frac\pi2]^n$ and,
for any $k\in\nn$, let $Q_k:=2k\pi+Q_0$.
Then, for any $k\in\nn$, we have
$$\int_{Q_k}f(x_1,\dots,x_n)\,dx_1\cdots dx_n
=\int_{2k\pi-\frac\pi2}^{2k\pi+\frac\pi2}\sin(x_1)\,dx_1
\cdots\int_{2k\pi-\frac\pi2}^{2k\pi+\frac\pi2}\sin(x_n)\,dx_n=0$$
and hence
\begin{align*}
\co(f;Q_0+2k\pi)&=\co(f;Q_k)=\fint_{Q_k}|f(x)-f_{Q_k}|\,dx=\fint_{Q_k}|f(x)|\,dx\\
&=\lf[\frac1\pi\int_{2k\pi-\frac\pi2}^{2k\pi+\frac\pi2}|\sin(t)|\,dt\r]^n
=\lf(\frac2\pi\r)^n,
\end{align*}
which can not tend to 0 as $k\to\fz$.
Thus, the above claim holds true,
which completes the proof of Corollary \ref{xMO-coro}.
\end{proof}

\begin{proposition}\label{xMO-replace}
Proposition \ref{CMO-char}{\rm (ii)} can be replaced by
\begin{itemize}
\item [{\rm(ii')}] $$\lim_{R\to\fz}\sup_{Q\cap Q(\vec 0_n,R)=\emptyset}\co(f;Q)=0,$$
\end{itemize}
where $Q(\vec 0_n,R)$ denotes the cube centered at $\vec 0_n$ with the side length $2R$.
However, ${\rm(ii)_2}$ of Theorem \ref{xMO-char}{\rm (ii)} can not be replaced by {\rm (ii')}.
\end{proposition}

\begin{proof}
Recall that Uchiyama \cite{Uchiyama78TohokuMathJ} stated Proposition \ref{CMO-char}
via (i), (ii) and (iii), while, in his proof, he proved that
Proposition \ref{CMO-char} with (ii) replaced by (ii') is true.
Indeed, this equivalence is a direct consequence of the following observation:
$$\rm (i)+(ii)+(iii)\ of\ Proposition\ \ref{CMO-char}
\Longrightarrow Proposition\ \ref{xMO-replace}(ii')
\Longrightarrow Proposition\ \ref{CMO-char}(ii).$$
To show that (ii)$_2$ of Theorem \ref{xMO-char}(ii) can not replaced
by Proposition \ref{xMO-replace}(ii'),
for simplicity, we only calculate a typical example in $\rr$.
Indeed, let $f(x):=\log (|x|)$ for any $x\in\rr$ with $|x|\ge1$,
and extend $f$ to $\rr$ smoothly.
Then $f\in C^1(\rr)\cap \BMO(\rr)$ and $\lim_{|x|\to\fz}f'(x)=0$,
which implies $f\in B_1(\rr)\subset\xMO(\rr)$.
On the other hand, for any $k\in\zz_+$ and interval $I_k:=[e^k,e^{k+1}]$,
we have
\begin{align*}
f_{I_k}&=\fint_{I_k}f(x)\,dx
=\frac1{e^{k+1}-e^{k}}\int_{e^{k}}^{e^{k+1}}\log(x)\,dx
=\frac{1}{(e-1)e^k}\lf[(k+1)e^{k+1}-e^{k+1}-(ke^k-e^k)\r]\\
&=\frac{1}{e-1}\lf[(k+1)e-e-(k-1)\r]
=\frac{1}{e-1}\lf(ke-k+1\r)=k+\frac1{e-1}
\end{align*}
and hence
\begin{align*}
\co(f;I_k)&=\fint_{I_k}|f(x)-f_{I_k}|\,dx
=\frac{1}{(e-1)e^k}\int_{e^{k}}^{e^{k+1}}\lf|\log(x)-\lf(k+\frac1{e-1}\r)\r|\,dx\\
&=\frac{1}{(e-1)e^k}\lf\{\int_{e^{k}}^{e^{k+\frac1{e-1}}}
  \lf[-\log(x)+\lf(k+\frac1{e-1}\r)\r]\,dx
  +\int_{e^{k+\frac1{e-1}}}^{e^{k+1}}\lf[\log(x)-\lf(k+\frac1{e-1}\r)\r]\,dx\r\}\\
&=\frac{1}{(e-1)e^k}\lf\{(k+1)e^{k+1}-e^{k+1}+ke^k-e^k
  -2\lf[\lf(k+\frac1{e-1}\r)e^{k+\frac1{e-1}}-e^{k+\frac1{e-1}}\r]\r.\\
&\lf.\qquad\qquad\qquad+\lf(k+\frac1{e-1}\r)\lf(2e^{k+\frac1{e-1}}-e^k-e^{k+1}\r)\r\}\\
&=\frac{1}{e-1}\lf\{(k+1)e-e+k-1
  -2\lf[\lf(k+\frac1{e-1}\r)e^{\frac1{e-1}}-e^{\frac1{e-1}}\r]
  +\lf(k+\frac1{e-1}\r)\lf(2e^{\frac1{e-1}}-1-e\r)\r\}\\
&=\frac{1}{e-1}\lf(-1+2e^{\frac1{e-1}}-\frac{e+1}{e-1}\r)
 =\frac{2}{e-1}\lf(e^{\frac1{e-1}}-\frac{e}{e-1}\r),
\end{align*}
which violates (ii') as long as $k$ satisfying $e^k>R$.
This finishes the proof of Proposition \ref{xMO-replace}.
\end{proof}

\begin{remark}\label{xMO-ques}
Observe that the counterexample in Proposition \ref{xMO-replace} is unbounded,
it is still unknown whether or not
(ii)$_1$ of Theorem \ref{xMO-char}(ii) $+$ Proposition \ref{xMO-replace}(ii')
is an equivalent characterization of $\MMO(\rn)$.
\end{remark}

\section{Proof of Theorem \ref{compact-thm}}\label{section3}

In this section, we prove Theorem \ref{compact-thm}
via several smooth truncated techniques.
Some of the ideas come from \cite{ClopCruz13AASFM}; see also \cite{TYY19}.
To begin with, we introduce the following smooth truncated function.
Let $\varphi_1\in C^\fz([0,\fz))$ satisfy
\begin{align}\label{phi1}
0\le\varphi_1\le1\quad\mathrm{and}\quad \varphi_1(x)=
\begin{cases}
1,\quad x\in[0,1],\\
0,\quad x\in[2,\fz).
\end{cases}
\end{align}
Moreover, for any $\eta\in(0,\fz)$, let
\begin{align}\label{Keta}
K_\eta(x,y,z):=K(x,y,z)\lf[1-\varphi_1\lf(\frac2\eta[|x-y|+|x-z|]\r)\r]
\end{align}
and, for any $f,\ g\in C_{\rm c}^\fz(\rn)$ and $x\notin\supp(f)\cap\supp(g)$,
\begin{align}\label{Teta}
T_\eta(f,g)(x):=\int_{\rn}K_\eta(x,y,z)f(y)g(z)\,dy\,dz.
\end{align}
Then
\begin{align}\label{Tetab}
\lf[b,T_\eta\r]_1(f,g)(x)=\int_{\rnn}[b(x)-b(y)]K_\eta(x,y,z)f(y)g(z)\,dy\,dz.
\end{align}

Recall that the \emph{bilinear Hardy--Littlewood maximal operator $\cm$}
is defined by setting, for any
$f,\ g\in L^1_{\loc}(\rn)$ and $x\in\rn$,
$$\cm (f,g)(x):=\sup_{\mathrm{cube}\,Q\ni x}\fint_Q|f(y)|\,dy \fint_Q|g(z)|\,dz,$$
where the supremum is taken over all the cubes $Q$ of $\rn$ containing $x$.
On $[b,T]_1$ and $[T_\eta,b]_1$, we have the following estimate via $\cm$.

\begin{lemma}\label{T-Teta}
There exists a positive constant $C$ such that, for any $b\in B_\fz(\rn)$,
$\eta\in(0,\infty)$, $f,\ g\in L^1_{\loc}(\rn)$ and $x\in\rn$,
$$\lf|[b,T]_1(f,g)(x)-\lf[b,T_\eta\r]_1(f,g)(x) \r|
\le C \eta \lf\|\nabla b\r\|_{L^\infty(\rn)} \cm (f,g)(x).$$
\end{lemma}

\begin{proof}
For any $x\in\rn$,
by \eqref{phi1} through \eqref{Tetab}, and \eqref{sizeregular}, we have
\begin{align*}
&\lf|[b,T]_1(f,g)(x)-\lf[b,T_\eta\r]_1(f,g)(x) \r|\\
&\quad=\lf|\int_{\rnn}[b(x)-b(y)]K(x,y,z)
      \varphi_1\lf(\frac2\eta[|x-y|+|x-z|]\r)f(y)g(z)\,dy\,dz\r|\\
&\quad\le \int_{|x-y|+|x-z|\le\eta}|b(x)-b(y)||K(x,y,z)||f(y)g(z)|\,dy\,dz\\
&\quad\ls \lf\|\nabla b\r\|_{L^\infty(\rn)}\sum_{j=0}^\infty
      \int_{\frac\eta{2^{j+1}}<|x-y|+|x-z|\leq\frac\eta{2^j}}
      \frac{|x-y|}{(|x-y|+|x-z|)^{2n}}|f(y)g(z)|\,dy\,dz\\
&\quad\ls \lf\|\nabla b\r\|_{L^\infty(\rn)}\sum_{j=0}^\infty
      \frac\eta{2^j}\frac1{(\frac\eta{2^{j+1}})^{2n}}
      \int_{Q(x,\frac\eta{2^j})\times Q(x,\frac\eta{2^j})}|f(y)g(z)|\,dy\,dz\\
&\quad\sim \lf\|\nabla b\r\|_{L^\infty(\rn)} \eta\sum_{j=0}^\infty
      \frac1{2^j}\fint_{Q(x,\frac\eta{2^j})}|f(y)|\,dy\fint_{Q(x,\frac\eta{2^j})}|g(z)|\,dz\\
&\quad\ls \lf\|\nabla b\r\|_{L^\infty(\rn)} \eta\sum_{j=0}^\infty
      \frac1{2^j}\cm(f,g)(x)
\sim \lf\|\nabla b\r\|_{L^\infty(\rn)} \eta \cm(f,g)(x),
\end{align*}
where $Q(x,\frac\eta{2^j})$ denotes the cube centered at $x$ with
the side length $2\frac\eta{2^j}$.
This finishes the proof of Lemma \ref{T-Teta}.
\end{proof}

We also need the following result on the relative compactness of a set
in weighted Lebesgue spaces, which is just \cite[Theorem 1]{XYY18}.

\begin{lemma}\label{FK}
Let $w$ be a weight on $\rn$. Assume that $w^{-1/(p_0-1)}$ is also a weight on $\rn$
for some $p_0\in(1,\fz)$. Let $p\in(0,\fz)$ and $\ce$ be a subset of $L_w^p(\rn)$.
Then $\ce$ is relatively compact in $L_w^p(\rn)$ if
the set $\ce$ satisfies the following three conditions:
\begin{itemize}
\item[{\rm(i)}] $\ce$ is bounded, namely,
$$\sup_{f\in\ce}\|f\|_{\Lpw}<\infty;$$

\item[{\rm(ii)}] $\ce$ uniformly vanishes at infinity, namely,
for any $\ez\in(0,\infty)$,
there exists some positive constant $A$ such that, for any $f\in\ce$,
$$\lf\|f\r\|_{L^p_w(\{|x|>A\})}<\ez;$$

\item[{\rm(iii)}] $\ce$ is uniformly equicontinuous, namely,
for any $\ez\in(0,\infty)$,
there exists some positive constant $\rho$ such that,
for any $f\in\ce$ and $t \in \rn$ with $|t|\in[0,\rho)$,
$$\|f(\cdot+t)-f(\cdot)\|_{\Lpw}<\ez.$$
\end{itemize}
\end{lemma}

\begin{remark}\label{FK-rem}
If $w$ is a classical $A_p(\rn)$ weight for some $p\in(1,\fz)$,
then the sufficiency of Lemma \ref{FK}
was first obtained in \cite[Theorem 5]{ClopCruz13AASFM}, which is needed
in the proof of Theorem \ref{compact-thm}.
\end{remark}

Let $\textbf{w}:=(w_1,w_2)\in \textbf{A}_{\textbf{p}}(\rn)$
and $w:=w_1^{p/p_1}w_2^{p/p_2}$ be as in Theorem \ref{compact-thm}.
From \cite[Theorem 3.6]{LOPTT09AM}, it follows that $w\in A_{2p}(\rn)$.
By this and Remark \ref{FK-rem}, we are now able to use Lemma \ref{FK}
to prove Theorem \ref{compact-thm} as follows.

\begin{proof}[Proof of Theorem \ref{compact-thm}]
Without loss of generality, it may suffice to prove this theorem
for the first entry $[T_\eta,b]_1$.
When $b\in \XMO(\rn)=\xMO(\rn)$ (see Theorem \ref{xMO-char}), from the definition of $\xMO(\rn)$,
we deduce that, for any $\ez\in(0,\infty)$,
there exists $b^{(\ez)}\in B_1(\rn)$ such that
$\|b-b^{(\ez)}\|_{\BMO(\rn)}<\ez.$
Then, by the boundedness of $[b,T]_1$ from
$\Lpwa\times \Lpwb$ to $\Lpw$
(see \cite[Theorem 3.18]{LOPTT09AM} and also \cite{BMMST19}
for more general results),
we obtain, for any $(f,g)\in\Lpwa\times\Lpwb$,
\begin{align*}
\lf\|[b,T]_1(f,g)-\lf[b^{(\ez)},T\r]_1(f,g)\r\|_{\Lpw}
&=\lf\|\lf[b-b^{(\ez)},T\r]_1(f,g)\r\|_{\Lpw}\\
&\lesssim\lf\|b-b^{(\ez)}\r\|_{\BMO(\rn)}\|f\|_{\Lpwa}\|g\|_{\Lpwb}\\
&\lesssim \ez \|f\|_{\Lpwa}\|g\|_{\Lpwb}.
\end{align*}
Moreover, using Lemma \ref{T-Teta} and the boundedness of $\cm$
from $\Lpwa\times \Lpwb$ to $\Lpw$ (\cite[Theorem 3.3]{LOPTT09AM}),
we conclude that
$$\lim_{\eta\to0}\lf\|[b,T]_1-\lf[b,T_\eta\r]_1\r\|_{\Lpwa\times \Lpwb\to\Lpw}=0,$$
where $\|[b,T]_1-[b,T_\eta]_1\|_{\Lpwa\times \Lpwb\to\Lpw}$ denotes
the operator norm of $[b,T]_1-[b,T_\eta]_1$ from $\Lpwa\times \Lpwb$ to $\Lpw$.
Thus, to prove that $[b,T]_1$ is compact for any $b\in\XMO(\rn)$,
by \cite[p.\,278, Theorem(iii)]{Yosida},
it suffices to show that $[b,T_\eta]_1$ is compact for any $b\in B_1(\rn)$
and $\eta\in(0,\fz)$ small enough.
To this end, by the definition of compact operators,
Lemma \ref{FK} and Remark \ref{FK-rem}(ii),
we know that it suffices to show that,
for any fixed $b\in B_1(\rn)$, $\eta\in(0,\fz)$ small enough
and $\ce\subset\Lpwa\times\Lpwb$ bounded,
$[b,T_\eta]_1\ce$ satisfies (i), (ii) and (iii) of Lemma \ref{FK}.
In what follows, we show these in order.

To begin with, we show that $[b,T_\eta]_1\ce$ satisfies Lemma \ref{FK}(i).
To this end, we first claim that $T_\eta$ is a Calder\'on--Zygmund operator,
namely, $K_\eta$ satisfies \eqref{sizeregular}.
Indeed, from $0\le\varphi\le1$ and hence $|K_\eta|\le|K|$,
it follows that the size condition [$\az:=\vec0_{3n}$
in \eqref{sizeregular}] holds true trivially.
When $|\az|=1$, without loss of generality, we may assume that $\az:=(1,\overbrace{0,\dots,0}^{3n-1\ {\rm times}})$.
Then, for any $x,\ y,\ z\in\rn$ with $x\neq y$ or $x\neq z$,
by \eqref{phi1} and \eqref{sizeregular}, we have
\begin{align*}
\lf|\frac{\partial}{\partial x_1} K_\eta(x,y,z)\r|
&\le\lf|\frac{\partial}{\partial x_1} K(x,y,z)\r|\lf|1-\varphi_1\lf(\frac2\eta[|x-y|+|x-z|]\r)\r|\\
&\quad+|K(x,y,z)|\lf|\frac{\partial}{\partial x_1}
  \lf[1-\varphi_1\lf(\frac2\eta[|x-y|+|x-z|]\r)\r]\r|\\
&\le\lf|\frac{\partial}{\partial x_1} K(x,y,z)\r|+|K(x,y,z)|\lf|\varphi_1'\lf(\frac2\eta[|x-y|+|x-z|]\r)\r|
  \frac2\eta\lf|\frac{\partial}{\partial x_1}(|x-y|+|x-z|)\r|\\
&\lesssim\frac1{(|x-y|+|x-z|)^{2n+1}}\\
&\quad+\frac{\frac2\eta(|x-y|+|x-z|)}{(|x-y|+|x-z|)^{2n+1}}
  \lf|\varphi_1'\lf(\frac2\eta[|x-y|+|x-z|]\r)\r|
  \lf|\frac{x_1-y_1}{|x-y|}+\frac{x_1-z_1}{|x-z|}\r|\\
&\lesssim\frac1{(|x-y|+|x-z|)^{2n+1}}.
\end{align*}
Therefore, $K_\eta$ satisfies \eqref{sizeregular} and hence
$T_\eta$ is a Calder\'on--Zygmund operator and,
moreover, the kernel constant is independent of $\eta$.
From this, $b\in\XMO(\rn)\subset\BMO(\rn)$,
the boundedness characterization of Calder\'on--Zygmund commutators
on weighted Lebesgue spaces (\cite[Theorem 3.18]{LOPTT09AM}),
and $(f,g)\in\ce$ bounded, we deduce that
$$\lf\|\lf[b,T_\eta\r]_1(f,g)\r\|_{\Lpw}\ls\|f\|_{\Lpwa}\|g\|_{\Lpwb}<\fz,$$
which implies that $[b,T_\eta]_1\ce$ satisfies Lemma \ref{FK}(i).

Next, we show that $[b,T_\eta]_1\ce$ satisfies Lemma \ref{FK}(ii). In what follows,
we use the symbol $E\gg D$ to denote that $E$ is much larger than $D$.
For fixed $A\gg1$ and fixed $x\in\rn$ with $|x|>A$,
we first split the truncated function $1-\varphi_1$ into the following two parts.
Let $\varphi_2,\ \varphi_3\in C^\fz([0,\fz))$ satisfy
\begin{align}\label{phi2}
0\le\varphi_2,\ \varphi_3\le1, \quad \varphi_2+\varphi_3=1-\varphi_1,
\end{align}
\begin{align}\label{phi3}
\varphi_2(x)=
\begin{cases}
1, \quad x\in\lf[2,\displaystyle\frac A2\r],\\
0, \quad x\in[0,1]\bigcup\lf[\displaystyle\frac A2+1,\fz\r)
\end{cases}
\mathrm{and}\quad \varphi_3(x)=
\begin{cases}
0, \quad x\in\lf[0,\displaystyle\frac A2\r],\\
1, \quad x\in\lf[\displaystyle\frac A2+1,\fz\r).
\end{cases}
\end{align}
Accordingly, for any $(f,g)\in\Lpwa\times\Lpwb$ and $x\in\rn$,
we split $[b,T_\eta]_1(f,g)(x)$ into the following three parts:
\begin{align}\label{L123}
\lf|\lf[b,T_\eta\r]_1(f,g)(x)\r|
&=\lf|\int_{\rnn}[b(x)-b(y)]K_\eta(x,y,z)f(y)g(z)\,dy\,dz \r|\\
&=\lf|\int_{\rnn}\nabla b(\xi)\cdot(x-y) K_\eta(x,y,z)\notag\r.\\
&\qquad\lf.\vphantom{\int_{\rnn}\nabla b(\xi)\cdot(x-y) K_\eta(x,y,z)}
  \times\lf[(\varphi_1+\varphi_2+\varphi_3)(|x-y|+|x-z|)\r]f(y)g(z)\,dy\,dz \r|\notag\\
&\le\int_{\rnn}|\nabla b(\xi)||x-y|\lf|K_\eta(x,y,z)\r|
  \varphi_1(|x-y|+|x-z|)|f(y)g(z)|\,dy\,dz \notag\\
&\quad+\int_{\rnn}|\nabla b(\xi)||x-y|\lf|K_\eta(x,y,z)\r|
  \varphi_2(|x-y|+|x-z|)|f(y)g(z)|\,dy\,dz \notag\\
&\quad+\int_{\rnn}|\nabla b(\xi)||x-y|\lf|K_\eta(x,y,z)\r|
  \varphi_3(|x-y|+|x-z|)|f(y)g(z)|\,dy\,dz \notag\\
&=:\mathrm{L}_1(x)+\mathrm{L}_2(x)+\mathrm{L}_3(x),\notag
\end{align}
where we applied the mean value theorem to $b(x)-b(y)$,
and $\xi$ is on the segment $\overline{xy}$ connecting $x$ and $y$.
We then estimate $\mathrm{L}_1(x)$ to $\mathrm{L}_3(x)$ in order.

To estimate $\mathrm{L}_1(x)$ as well as $\|\mathrm{L}_1\|_{L^p_w(\{|x|>A\})}$,
we notice that $\xi\in\overline{xy}$ and hence
\begin{align}\label{xixy}
|x-\xi|\le|x-y|.
\end{align}
Meanwhile, by \eqref{phi1}, we know that $\supp(\varphi_1)\subset[0,2]$ and hence
\begin{align}\label{le2}
|x-y|+|x-z|\le2.
\end{align}
From \eqref{xixy}, \eqref{le2} and $|x|>A\gg1$, it follows that
\begin{align}\label{xiA2}
|\xi|\ge|x|-|x-\xi|\ge|x|-|x-y|\ge|x|-(|x-y|+|x-z|)\ge|x|-2>\frac A2.
\end{align}
By \eqref{xiA2}, \eqref{sizeregular}, $0\le\varphi\le1$
and hence $|K_\eta|\le|K|$, we conclude that,
for any $(f,g)\in\Lpwa\times\Lpwb$ and $x\in\rn$,
\begin{align}\label{L1toJ1}
\mathrm{L}_1(x)
\ls&\,\sup_{|\xi|>\frac A2}|\nabla b(\xi)|
 \int_{\rnn}\frac{|x-y|\varphi_1(|x-y|+|x-z|)}{(|x-y|+|x-z|)^{2n}}|f(y)g(z)|\,dy\,dz\\
\ls&\,\sup_{|\xi|>\frac A2}|\nabla b(\xi)|
 \int_{\rnn}\frac{\varphi_1(|x-y|+|x-z|)}{(|x-y|+|x-z|)^{2n-1}}|f(y)g(z)|\,dy\,dz\notag\\
:=&\,\sup_{|\xi|>\frac A2}|\nabla b(\xi)|
 \int_{\rnn}\mathcal{K}_1(x,y,z)|f(y)g(z)|\,dy\,dz
=:\sup_{|\xi|>\frac A2}|\nabla b(\xi)|\mathrm{J}_1(x).\notag
\end{align}
We now claim that $\mathrm{J}_1$ is a Calder\'on--Zygmund operator,
namely, $\mathcal{K}_1$ satisfies \eqref{sizeregular}.
Indeed, by \eqref{phi1}, we have, for any $x,\ y,\ z\in\rn$ with $x\neq y$ or $x\neq z$,
\begin{align*}
|\mathcal{K}_1(x,y,z)|&=\frac{\varphi_1(|x-y|+|x-z|)}{(|x-y|+|x-z|)^{2n-1}}
=\frac{(|x-y|+|x-z|)\varphi_1(|x-y|+|x-z|)}{(|x-y|+|x-z|)^{2n}}\\
&\le\frac{2}{(|x-y|+|x-z|)^{2n}}
\end{align*}
and hence the size condition [$\az=\vec0_{3n}$ in \eqref{sizeregular}] holds true.
When $|\az|=1$, without loss of generality, we may assume that $\az:=(1,\overbrace{0,\dots,0}^{3n-1\ {\rm times}})$.
Then, for any $x,\ y,\ z\in\rn$ with $x\neq y$ or $x\neq z$, we obtain
\begin{align*}
\lf|\frac{\partial}{\partial x_1}\ck_1(x,y,z)\r|
&\ls\frac{|\varphi_1'(|x-y|+|x-z|)(\frac{x_1-y_1}{|x-y|}+\frac{x_1-z_1}{|x-z|})
  (|x-y|+|x-z|)^2|}{(|x-y|+|x-z|)^{2n+1}}\\
&\quad+\frac{|\varphi_1(|x-y|+|x-z|)(2n-1)(|x-y|+|x-z|)(\frac{x_1-y_1}{|x-y|}+\frac{x_1-z_1}{|x-z|})|}
  {(|x-y|+|x-z|)^{2n+1}}\\
&\ls\frac{\|\varphi_1'\|_{L^\fz(\rn)}+1}
  {(|x-y|+|x-z|)^{2n+1}}
 \ls\frac1{(|x-y|+|x-z|)^{2n+1}}.
\end{align*}
Therefore, $\ck_1$ satisfies \eqref{sizeregular} and hence $\mathrm{J}_1$ is a
Calder\'on--Zygmund operator, which shows the above claim.
From this claim, the boundedness of Calder\'on--Zygmund operators
on weighted Lebesgue spaces
(\cite[Corollary 3.9]{LOPTT09AM}) and \eqref{L1toJ1}, we deduce that, for any $(f,g)\in\Lpwa\times\Lpwb$,
\begin{align}\label{L1ok}
\|\mathrm{L}_1\|_{L^p_w(\{|x|>A\})}\ls\sup_{|\xi|>\frac A2}|\nabla b(\xi)|\|\mathrm{J}_1\|_{\Lpw}
\ls\sup_{|\xi|>\frac A2}|\nabla b(\xi)|\|f\|_{\Lpwa}\|g\|_{\Lpwb}.
\end{align}

Next, we estimate $\mathrm{L}_2(x)$ as well as
$\|\mathrm{L}_2\|_{L^p_w(\{|x|>A\})}$.
By \eqref{phi2} and \eqref{phi3}, we know that
$\supp(\varphi_2)\subset[1,\frac A2]$ and hence
\begin{align}\label{leA2}
|x-y|+|x-z|\le \frac A2.
\end{align}
From \eqref{xixy}, \eqref{le2} and $|x|>A\gg1$, we deduce that
\begin{align*}
|\xi|\ge|x|-|x-\xi|\ge|x|-|x-y|\ge|x|-(|x-y|+|x-z|)\ge|x|-\frac A2>\frac A2,
\end{align*}
and hence \eqref{xiA2} still holds true.
By \eqref{xiA2}, \eqref{decay}, $0\le\varphi\le1$
and hence $|K_\eta|\le|K|$, we further conclude that,
for any $(f,g)\in\Lpwa\times\Lpwb$ and $x\in\rn$,
\begin{align}\label{L2toJ2}
\mathrm{L}_2(x)
\ls&\,\sup_{|\xi|>\frac A2}|\nabla b(\xi)|
 \int_{\rnn}\frac{|x-y|\varphi_2(|x-y|+|x-z|)}{(|x-y|+|x-z|)^{2n+2}}|f(y)g(z)|\,dy\,dz\\
\ls&\,\sup_{|\xi|>\frac A2}|\nabla b(\xi)|
 \int_{\rnn}\frac{\varphi_2(|x-y|+|x-z|)}{(|x-y|+|x-z|)^{2n+1}}|f(y)g(z)|\,dy\,dz\notag\\
:=&\,\sup_{|\xi|>\frac A2}|\nabla b(\xi)|
 \int_{\rnn}\mathcal{K}_2(x,y,z)|f(y)g(z)|\,dy\,dz
=:\sup_{|\xi|>\frac A2}|\nabla b(\xi)|\mathrm{J}_2(x).\notag
\end{align}
We now claim that $\mathrm{J}_2$ is also a Calder\'on--Zygmund operator,
namely, $\mathcal{K}_2$ satisfies \eqref{sizeregular}.
Indeed, by \eqref{phi2} and \eqref{phi3}, we obtain,
for any $x,\ y,\ z\in\rn$ with $x\neq y$ or $x\neq z$,
$$|\mathcal{K}_2(x,y,z)|=\frac{\varphi_2(|x-y|+|x-z|)}{(|x-y|+|x-z|)^{2n+1}}
\le\frac{1}{(|x-y|+|x-z|)^{2n}}$$
and hence the size condition [$\az=\vec0_{3n}$
in \eqref{sizeregular}] holds true.
When $|\az|=1$, without loss of generality,
we may assume that $\az:=(1,\overbrace{0,\dots,0}^{3n-1\ {\rm times}})$.
Then, for any $x,\ y,\ z\in\rn$ with $x\neq y$ or $x\neq z$,
it holds true that
\begin{align*}
\lf|\frac{\partial}{\partial x_1}\ck_2(x,y,z)\r|
&\ls\frac{|\varphi_2'(|x-y|+|x-z|)(\frac{x_1-y_1}{|x-y|}+\frac{x_1-z_1}{|x-z|})|}
  {(|x-y|+|x-z|)^{2n+1}}\\
&\quad+\frac{|\varphi_2(|x-y|+|x-z|)(2n+1)(|x-y|+|x-z|)^{-1}
  (\frac{x_1-y_1}{|x-y|}+\frac{x_1-z_1}{|x-z|})|}{(|x-y|+|x-z|)^{2n+1}}\\
&\ls\frac{\|\varphi_2'\|_{L^\fz(\rn)}+1}{(|x-y|+|x-z|)^{2n+1}}
 \ls\frac1{(|x-y|+|x-z|)^{2n+1}},
\end{align*}
where the implicated positive constant is independent of $A$.
Therefore, $\ck_2$ satisfies \eqref{sizeregular}, and hence $\mathrm{J}_2$ is a
Calder\'on--Zygmund operator, which shows the above claim holds true.
Using this claim, the boundedness of Calder\'on--Zygmund operators
on weighted Lebesgue spaces
(\cite[Corollary 3.9]{LOPTT09AM}) and \eqref{L2toJ2},
we find that, for any $(f,g)\in\Lpwa\times\Lpwb$,
\begin{align}\label{L2ok}
\|\mathrm{L}_2\|_{L^p_w(\{|x|>A\})}\ls\sup_{|\xi|>\frac A2}|\nabla b(\xi)|\|\mathrm{J}_2\|_{\Lpw}
\ls\sup_{|\xi|>\frac A2}|\nabla b(\xi)|\|f\|_{\Lpwa}\|g\|_{\Lpwb}
\end{align}
with the implicit positive constant independent of $A$.

Now, we estimate $\mathrm{L}_3(x)$ as well as $\|\mathrm{L}_3\|_{L^p_w(\{|x|>A\})}$.
From \eqref{decay}, $0\le\varphi\le1$ and hence $|K_\eta|\le|K|$,
we deduce that, for any $(f,g)\in\Lpwa\times\Lpwb$ and $x\in\rn$,
\begin{align}\label{L3toJ3}
\mathrm{L}_3(x)
\ls&\,\|\nabla b\|_{L^\fz(\rn)}
 \int_{\rnn}\frac{|x-y|\varphi_3(|x-y|+|x-z|)}{(|x-y|+|x-z|)^{2n+2}}|f(y)g(z)|\,dy\,dz\\
\ls&\,\|\nabla b\|_{L^\fz(\rn)}
 \int_{\rnn}\frac{\varphi_3(|x-y|+|x-z|)}{(|x-y|+|x-z|)^{2n+1}}|f(y)g(z)|\,dy\,dz\notag\\
:=&\,\|\nabla b\|_{L^\fz(\rn)}
 \int_{\rnn}\mathcal{K}_3(x,y,z)|f(y)g(z)|\,dy\,dz
=:\|\nabla b\|_{L^\fz(\rn)}\mathrm{J}_3(x).\notag
\end{align}
We now claim that $\mathrm{J}_3$ is also a Calder\'on--Zygmund operator,
namely, $\mathcal{K}_3$ satisfies \eqref{sizeregular}
with the kernel constant $O(A^{-1})$.
Indeed, by \eqref{phi2} and \eqref{phi3}, we have,
for any $x,\ y,\ z\in\rn$ with $x\neq y$ or $x\neq z$,
$$|\mathcal{K}_3(x,y,z)|=\frac{\varphi_2(|x-y|+|x-z|)}{(|x-y|+|x-z|)^{2n+1}}
\le\frac{2}{A(|x-y|+|x-z|)^{2n}}$$
and hence the size condition [$\az=\vec0_{3n}$
in \eqref{sizeregular}] holds true with the kernel constant $O(A^{-1})$.
When $|\az|=1$, without loss of generality,
we may assume that $\az:=(1,\overbrace{0,\dots,0}^{3n-1\ {\rm times}})$.
Then, for any $x,\ y,\ z\in\rn$ with $x\neq y$ or $x\neq z$,
\begin{align*}
\lf|\frac{\partial}{\partial x_1}\ck_3(x,y,z)\r|
&\ls\frac{|\varphi_3'(|x-y|+|x-z|)(\frac{x_1-y_1}{|x-y|}+\frac{x_1-z_1}{|x-z|})|}
  {(|x-y|+|x-z|)^{2n+1}}\\
&\quad+\frac{|\varphi_3(|x-y|+|x-z|)(2n+1)(|x-y|+|x-z|)^{-1}(\frac{x_1-y_1}{|x-y|}+\frac{x_1-z_1}{|x-z|})|}
  {(|x-y|+|x-z|)^{2n+1}}\\
&\ls\frac{\|\varphi_3'\|_{L^\fz(\rn)}+(\frac A2)^{-1}}
  {(|x-y|+|x-z|)^{2n+1}}
 \ls\frac{1}{A(|x-y|+|x-z|)^{2n+1}}.
\end{align*}
Thus, $\ck_3$ satisfies \eqref{sizeregular} with the kernel constant $O(A^{-1})$
and hence $\mathrm{J}_3$ is a Calder\'on--Zygmund operator,
which shows the above claim.
By this claim, the boundedness of Calder\'on--Zygmund operators
on weighted Lebesgue spaces
(\cite[Corollary 3.9]{LOPTT09AM}) and \eqref{L3toJ3}, we conclude that,
for any $(f,g)\in\Lpwa\times\Lpwb$,
\begin{align}\label{L3ok}
\|\mathrm{L}_3\|_{L^p_w(\{|x|>A\})}\ls\|\nabla b\|_{L^\fz(\rn)}\|\mathrm{J}_2\|_{\Lpw}
\ls\frac{\|\nabla b\|_{L^\fz(\rn)}}{A}\|f\|_{\Lpwa}\|g\|_{\Lpwb}.
\end{align}

To sum up, for any given $\ez\in(0,\fz)$,
there exists a positive constant $A$ large enough such that
both
$$\sup_{|\xi|>\frac A2}|\nabla b(\xi)|\ls \ez \quad {\rm and}\quad
\frac{\|\nabla b\|_{L^\fz(\rn)}}{A}\ls \ez$$
hold true, where the implicit positive constants are independent of $A$ and $\ez$.
From this, \eqref{L123}, \eqref{L1ok}, \eqref{L2ok}, \eqref{L3ok}, $b\in B_1(\rn)$
and $(f,g)\in\ce$ bounded, we deduce that
$$\lf\|\lf[b,T_\eta\r]_1(f,g)\r\|_{L^p_w(\{|x|>A\})}
\le\sum_{k=1}^3\|\mathrm{L}_k\|_{L^p_w(\{|x|>A\})}<\ez,$$
which implies that $[b,T_\eta]_1\ce$ satisfies Lemma \ref{FK}(ii).

It remains to prove that $[b,T_\eta]_1\ce$ also satisfies Lemma \ref{FK}(iii).
Recall that $\eta$ is a fixed positive constant small enough.
Let $t\in\rn$ satisfy $|t|\in(0,\eta/8)$.
Then, for any $(f,g)\in\Lpwa\times\Lpwb$ and $x\in\rn$, we have
\begin{align}\label{L45}
&\lf[b,T_\eta\r]_1(f,g)(x)-\lf[b,T_\eta\r]_1(f,g)(x+t)\\
&\quad=\int_{\rnn}[b(x)-b(y)]K_\eta(x,y,z)f(y)g(z)\,dy\,dz\notag\\
&\qquad -\int_{\rnn}[b(x+t)-b(y)]K_\eta(x+t,y,z)f(y)g(z)\,dy\,dz\notag\\
&\quad=[b(x)-b(x+t)]\int_{\rnn}K_\eta(x,y,z)f(y)g(z)\,dy\,dz\notag\\
&\qquad +\int_{\rnn}[b(x+t)-b(y)]\lf[K_\eta(x,y,z)-K_\eta(x+t,y,z)\r]f(y)g(z)\,dy\,dz\notag\\
&\quad=:\mathrm{L}_4(x)+\mathrm{L}_5(x).\notag
\end{align}

Observe that we have shown that $K_\eta$ satisfies \eqref{sizeregular}
with the kernel constant independent of $\eta$
[see the proof of $[b,T_\eta]_1\ce$ satisfying Lemma \ref{FK}(i) above].
From this, the mean value theorem
and the boundedness of Calder\'on--Zygmund operators
on weighted Lebesgue spaces (\cite[Corollary 3.9]{LOPTT09AM}),
it follows that, for any $(f,g)\in\Lpwa\times\Lpwb$,
\begin{align}\label{L4ok}
\|\mathrm{L}_4\|_{\Lpw}\ls |t|\|\nabla b\|_{L^\fz}\|f\|_{\Lpwa}\|g\|_{\Lpwb}.
\end{align}

To estimate $\mathrm{L}_5$, we first observe that,
for any $x,\ y,\ z,\ t\in\rn$ with
$|x-y|+|x-z|<\frac\eta4$ and $|t|<\frac\eta8$,
$$\varphi_1\lf(\frac2\eta[|x-y|+|x-z|]\r)=0
=\varphi_1\lf(\frac2\eta[|x+t-y|+|x+t-z|]\r)$$
and hence
\begin{align}\label{K=0=K}
K_\eta(x,y,z)=0=K_\eta(x+t,y,z).
\end{align}
Besides, for any $x,\ y,\ z,\ t\in\rn$ with
$|x-y|+|x-z|\ge\frac\eta4$ and $|t|<\frac\eta{16}$,
we have
$$|t|\le\frac14(|x-y|+|x-z|).$$
This, together with the mean value theorem and \eqref{sizeregular}, implies that
\begin{align}\label{Kx-Kx+t}
\lf|K_\eta(x,y,z)-K_\eta(x+t,y,z) \r|
&=\lf|t\nabla_x K_\eta(\zeta,y,z)\r|
\ls \frac{|t|}{(|\zeta-y|+|\zeta-z|)^{2n+1}}\\
&\ls \frac{|t|}{(|x-y|+|x-z|-2|x-\zeta|)^{2n+1}}\notag\\
&\ls \frac{|t|}{(|x-y|+|x-z|-2|t|)^{2n+1}}
\ls \frac{|t|}{(|x-y|+|x-z|)^{2n+1}}\notag,
\end{align}
where $\nabla_x K_\eta(\cdot,y,z)$ denotes the gradient of
$K_\eta(\cdot,y,z)$ on the first variable with $y$, $z$ fixed,
and $\zeta$ is in the segment connecting $x$ and $x+t$.
By \eqref{L45}, the mean value theorem, \eqref{K=0=K} and \eqref{Kx-Kx+t},
we have, for any $(f,g)\in\Lpwa\times\Lpwb$ and $x\in\rn$,
\begin{align*}
|\mathrm{L}_5(x)|&\ls|t|^2\|\nabla b\|_{L^\fz(\rn)}\int_{|x-y|+|x-z|>\frac\eta4}
  \frac{|f(y)g(z)|}{(|x-y|+|x-z|)^{2n+1}}\,dy\,dz\\
&\ls|t|^2\|\nabla b\|_{L^\fz(\rn)}\sum_{k=0}^\fz\frac1{(2^k\eta)^{2n+1}}
  \int_{2^k\frac\eta4<|x-y|+|x-z|\le2^{k+1}\frac\eta4}f(y)g(z)\,dy\,dz\notag\\
&\ls |t|^2\|\nabla b\|_{L^\fz(\rn)}\sum_{k=0}^\fz\frac1{2^k\eta}
  \fint_{Q(x,2^{k+1}\frac\eta4)}|f(y)|\,dy\fint_{Q(x,2^{k+1}\frac\eta4)}|g(z)|\,dz\notag\\
&\ls \frac{|t|^2}{\eta}\|\nabla b\|_{L^\fz(\rn)} \cm(f,g)(x)\notag
\end{align*}
and hence, by the boundedness of $\cm$
from $\Lpwa\times \Lpwb$ to $\Lpw$ (\cite[Theorem 3.3]{LOPTT09AM}),
we further conclude that, for any $(f,g)\in\Lpwa\times\Lpwb$,
\begin{align}\label{L5ok}
\|\mathrm{L}_5\|\ls |t|^2\|\nabla b\|_{L^\fz}\|f\|_{\Lpwa}\|g\|_{\Lpwb}.
\end{align}
Combining \eqref{L45}, \eqref{L4ok} and \eqref{L5ok}, we conclude that,
for any $b\in B_1(\rn)$ and $(f,g)\in\ce$ bounded,
$$\lim_{|t|\to0}\lf\|\lf[b,T_\eta\r]_1(f,g)(\cdot)-\lf[b,T_\eta\r]_1(f,g)(\cdot+t) \r\|_{\Lpw}=0,$$
which implies that $[b,T_\eta]_1\ce$ also satisfies Lemma \ref{FK}(iii).
Thus, $[b,T_\eta]_1\ce$ satisfies (i), (ii) and (iii) of Lemma \ref{FK} and hence
$[b,T_\eta]_1$ is a compact operator for any $b\in B_1(\rn)$
and $\eta\in(0,\fz)$ small enough.
This finishes the proof of Theorem \ref{compact-thm}.
\end{proof}

\noindent\textbf{Acknowledgements}.
Jin Tao and Dachun Yang would like to express their sincerely thanks to
Professor Loukas Grafakos for his very helpful discussion on the subject of this article.

\bigskip

\noindent Jin Tao, Qingying Xue, Dachun Yang (Corresponding author) and Wen Yuan

\medskip

\noindent Laboratory of Mathematics and Complex Systems
(Ministry of Education of China),
School of Mathematical Sciences, Beijing Normal University,
Beijing 100875, People's Republic of China

\smallskip

\noindent{\it E-mails:} \texttt{jintao@mail.bnu.edu.cn} (J. Tao)

\noindent\phantom{{\it E-mails:}} \texttt{qyxue@bnu.edu.cn} (Q. Xue)

\noindent\phantom{{\it E-mails:}} \texttt{dcyang@bnu.edu.cn} (D. Yang)

\noindent\phantom{{\it E-mails:}} \texttt{wenyuan@bnu.edu.cn} (W. Yuan)


\begin{thebibliography}{99}

\bibitem{BDMT15}
A. B\'enyi, W. Dami\'an, K. Moen and R. H. Torres,
Compact bilinear commutators: the weighted case,
Michigan Math. J. 64 (2015), 39--51.

\vspace{-0.3cm}

\bibitem{BMMST19}
A. B\'enyi, J. M. Martell, K. Moen, E. Stachura and R. H. Torres,
Boundedness results for commutators with BMO functions
via weighted estimates: a comprehensive approach,
Math. Ann. 376 (2020), 61--102.

\vspace{-0.3cm}

\bibitem{BT03}
A. B\'enyi and R. H. Torres,
Symbolic calculus and the transposes of bilinear pseudodifferential operators,
Comm. Partial Differential Equations 28 (2003), 1161--1181.


\vspace{-0.3cm}

\bibitem{BT}
A. B\'enyi and R. H. Torres,
Compact bilinear operators and commutators,
Proc. Amer. Math. Soc.
141 (2013), 3609--3621.

\vspace{-0.3cm}

\bibitem{BD13}
T. A. Bui and X. T. Duong,
Weighted norm inequalities for multilinear operators and applications to multilinear Fourier multipliers,
Bull. Sci. Math. 137 (2013), 63--75.

\vspace{-0.3cm}



\bibitem{CCHTW}
L. Chaffee, P. Chen, Y. Han, R. H. Torres and L. A. Ward,
Characterization of compactness of commutators of bilinear singular integral operators,
{Proc. Amer. Math. Soc.}
146 (2018), 3943--3953.

\vspace{-0.3cm}

\bibitem{CDLW19}
P. Chen, X. Duong, J. Li and Q. Wu,
Compactness of Riesz transform commutator on stratified Lie groups,
J. Funct. Anal. 277 (2019), 1639--1676.

\vspace{-0.3cm}


\bibitem{ClopCruz13AASFM}
A. Clop and V. Cruz,
Weighted estimates for Beltrami equations,
Ann. Acad. Sci. Fenn. Math.
38 (2013), 91--113.

\vspace{-0.3cm}

\bibitem{CM78}
R. R. Coifman and Y. Meyer,
Commutateurs d' int\'egrales singuli\`eres et op\'erateurs multilin\'eaires,
Ann. Inst. Fourier (Grenoble) 28 (1978), 177--202.

\vspace{-0.3cm}


\bibitem{CoifmanRochbergWeiss76AnnMath}
R. R. Coifman, R. Rochberg and G. Weiss,
Factorization theorems for Hardy spaces in several variables,
Ann. of Math. (2) 103 (1976), 611--635.

\vspace{-0.3cm}


\bibitem{CO}
H. O. Cordes,
On compactness of commutators of multiplications and convolutions, and boundedness of pseudodifferential operators,
J. Funct. Anal. 18 (1975), 115--131.

\vspace{-0.3cm}
\bibitem{DMX}
Y. Ding, T. Mei and Q. Xue,
Compactness of maximal commutators of bilinear Calder\'on--Zygmund
singular integral operators, in: Some Topics in Harmonic Analysis and Applications, 163--175,
Adv. Lect. Math. (ALM), 34, Int. Press, Somerville, MA, 2016.

\vspace{-0.3cm}

\bibitem{GT02}
L. Grafakos and R. H. Torres,
Multilinear  Calder\'on--Zygmund theory,
Adv. Math. 165 (2002), 124--164.

\vspace{-0.3cm}

\bibitem{GMNT}
L. Grafakos, A. Miyachi, H. Van Nguyen and N. Tomita,
Multilinear Fourier multipliers with minimal Sobolev regularity, II,
J. Math. Soc. Japan 69 (2017), 529--562.

\vspace{-0.3cm}

\bibitem
{GMN} L. Grafakos and H. Van Nguyen, Multilinear Fourier multipliers with minimal Sobolev
regularity, I, Colloq. Math. 144 (2016), 1--30.

\vspace{-0.3cm}

\bibitem{H1}
{G. Hu,}
Compactness of the commutator of bilinear Fourier multiplier operator,
Taiwanese J. Math.
18  (2014), 661--675.

\vspace{-0.3cm}
\bibitem{J}
S. Janson,
Mean oscillation and commutators of singular integral operators,
Ark. Mat. 16 (1978), 263--270.

\vspace{-0.3cm}

\bibitem{LOPTT09AM}
A. Lerner, S. Ombrosi, C. P\'erez, R. H. Torres and R. Trujillo-Gonz\'alez,
New maximal functions and multiple weights for the multilinear Calder\'on--Zygmund theory,
Adv. Math. 220 (2009), 1222--1264.

\vspace{-0.3cm}

\bibitem{LNWW20}
J. Li, T. T. Nguyen, L. A. Ward and B. D. Wick,
The Cauchy integral, bounded and compact commutators,
Studia Math. 250 (2020), 193--216.

\vspace{-0.3cm}

\bibitem{L20}
K. Li, Multilinear commutators in the two-weight setting,
arXiv: 2006.09071.

\vspace{-0.3cm}

\bibitem{PT}
C. P\'erez and R. H. Torres,
Sharp maximal function estimates for multilinear singular integrals,
in: Harmonic Analysis at Mount Holyoke (South Hadley, MA, 2001), 323--331,
Contemp. Math. 320, Amer. Math. Soc. Providence, RI, 2003.

\vspace{-0.3cm}

\bibitem{Sarason75}
D. Sarason,
Functions of vanishing mean oscillation,
Trans. Amer. Math. Soc. 207 (1975), 391--405.

\vspace{-0.3cm}
\bibitem{Tang}
L. Tang,
Weighted estimates for vector-valued commutators of multilinear operators,
Proc. Roy. Soc. Edinburgh Sect. A
138 (2008), 897--922.

\vspace{-0.3cm}

\bibitem{TYY19}
J. Tao, Da. Yang and Do. Yang,
Beurling--Ahlfors commutators on weighted Morrey spaces
and applications to Beltrami equations,
Potential Anal. (2020), https://doi.org/10.1007/ s11118-019-09814-7.

\vspace{-0.3cm}

\bibitem{TX19}
R. H. Torres and Q. Xue,
On compactness of commutators of multiplication and bilinear
pesudodifferential operators and a new subspace of $\BMO$,
Rev. Mat. Iberoam. 36 (2020), 939--956.

\vspace{-0.3cm}

\bibitem{TXY}
R. H. Torres, Q. Xue  and J. Yan,
Compact bilinear commutators: the quasi-Banach space case,
J. Anal.  26 (2018), 227--234.

\vspace{-0.3cm}

\bibitem{Uchiyama78TohokuMathJ}
A. Uchiyama,
On the compactness of operators of Hankel type,
T\^ohoku Math. J. (2) 30 (1978),
163--171.

\vspace{-0.3cm}

\bibitem{XYY18}
Q. Xue, K. Yabuta and J. Yan,
Weighted Fr\'echet--Kolmogorov theorem and compactness of
vector-valued multilinear operators,
arXiv: 1806.06656


\vspace{-0.3cm}

\bibitem{Yosida}
K. Yosida,
Functional Analysis, Reprint of the sixth (1980) edition,
Classics in Mathematics, Springer-Verlag, Berlin, 1995.

\end{thebibliography}
\end{document}